\documentclass{article}

\usepackage{fullpage,amsmath,amssymb,amsthm,amsfonts,enumerate,textcomp, eurosym, titling, epsfig,epstopdf, esint, authblk,pinlabel}

\makeatletter
\def\grd@save@target#1{%
  \def\grd@target{#1}}
\def\grd@save@start#1{%
  \def\grd@start{#1}}

\begin{document}

\makeatother

\title{Fundamental shadow links realized as links in $S^3$}
\author{Sanjay Kumar}
\date{\vspace{-5ex}}

\newcommand{\Addresses}{{
  \bigskip
  \footnotesize

  \textsc{Department of Mathematics, Michigan State University,
    East Lansing, MI, 48824, USA}\par\nopagebreak
  \textit{E-mail address}: \texttt{kumarsa4@msu.edu}

}}

\maketitle

\begin{abstract}
We use purely topological tools to construct several infinite families of hyperbolic links in the $3$-sphere  that  satisfy the Turaev-Viro invariant  volume conjecture posed by Chen and Yang.
To show that our links satisfy the volume conjecture, we prove that each has complement homeomorphic to the complement of a fundamental shadow link. These are links  
in connected sums of copies of $S^2 \times S^1$ for which the conjecture is known due to Belletti, Detcherry, Kalfagianni, and Yang.  Our methods also
verify the conjecture for several hyperbolic links with crossing number less than twelve. In addition, we show  that every link in the $3$-sphere is a sublink of a link that satisfies the conjecture. 

As an application of our results, we extend the class of known examples that satisfy the AMU conjecture on quantum representations of surface mapping class groups. For example, we give explicit elements in
the mapping class group of a genus $g$ surface with four boundary components for any $g$. 
For this, we use  techniques developed by Detcherry and Kalfagianni which relate the Turaev-Viro invariant volume conjecture to the AMU conjecture.  
\end{abstract}

\newtheorem{innercustomgeneric}{\customgenericname}
\providecommand{\customgenericname}{}
\newcommand{\newcustomtheorem}[2]{%
  \newenvironment{#1}[1]
  {%
   \renewcommand\customgenericname{#2}%
   \renewcommand\theinnercustomgeneric{##1}%
   \innercustomgeneric
  }
  {\endinnercustomgeneric}
}

\newcustomtheorem{customthm}{Theorem}
\newcustomtheorem{customlemma}{Lemma}
\newcustomtheorem{customprop}{Proposition}
\newcustomtheorem{customconjecture}{Conjecture}
\newcustomtheorem{customcor}{Corollary}

\theoremstyle{plain}
\newtheorem*{ack*}{Acknowledgements}
\newtheorem{thm}{Theorem}[section]
\newtheorem{lem}[thm]{Lemma}
\newtheorem{prop}[thm]{Proposition}
\newtheorem{cor}[thm]{Corollary}
\newtheorem{predefinition}[thm]{Definition}
\newtheorem{conjecture}[thm]{Conjecture}
\newtheorem{preremark}[thm]{Remark}
\newenvironment{remark}%
  {\begin{preremark}\upshape}{\end{preremark}}
 \newenvironment{definition}%
  {\begin{predefinition}\upshape}{\end{predefinition}}

\newtheorem{ex}[thm]{Example}

\renewcommand\thesubsection{\thesection\Alph{subsection}}

\section{Introduction}\label{1}

The Turaev-Viro invariants $TV_q(M)$ of a compact $3$-manifold $M$ \cite{TurV92} are real numbers  that depend on 
an integer  $r \geq 3,$  and a $2r$-th root of unity $q.$ In this paper, we will be  concerned with specifically the $SO(3)$-version of the Turaev-Viro invariant with $r \geq 3$, odd, and $q=\exp\left(\frac{2 \pi i}{r}\right)$.

In \cite{CheY18}, Chen and Yang provided computational evidence and proposed the following conjecture relating the volume of a manifold $M$, denoted as $Vol(M)$, to the Turaev-Viro invariants.

\begin{conjecture}[\cite{CheY18}] \label{VolC}
Let  $M$ be a hyperbolic $3$-manifold, either closed, with cusps, or compact with totally geodesic boundary,  and $q=\exp\left(\frac{2 \pi i}{r}\right).$ Then
$$\lim_{r \to \infty} \frac{2 \pi}{r} \log| TV_q(M)| = Vol(M),$$
as  $r$ ranges along  the odd natural numbers.
\end{conjecture}

The conjecture is a 3-manifold analogue of the original volume conjecture stated by Kashaev \cite{Kas97} which was reformulated in terms of the colored Jones polynomial by H. Murakami and J. Murakami \cite{MurM01}, and asserts that the asymptotics of the $N$-colored Jones polynomial of a hyperbolic knot at a particular root of unity approaches the volume of its complement. 

 Conjecture \ref{VolC}  was proved for the complement of the Borromean rings and the complement of the figure-eight knot by Detcherry, Kalfagianni, and Yang in \cite{DetKY18}. It was verified by Ohtsuki that the closed hyperbolic $3$-manifolds arising from integral Dehn surgeries on the figure-eight knot satisfy the conjecture  in \cite{Oht18}. More recently,  Wong and Yang extended Ohtsuki's result to closed hyperbolic $3$-manifolds obtained from rational surgeries along the figure-eight knot in \cite{WonY20}. 
Belletti, Detcherry, Kalfagianni, and Yang verified the conjecture
for the complements of  fundamental shadow links in \cite{BelDKY}. This  is an infinite class of hyperbolic links in connected sums of copies of $S^2 \times S^1$ which was 
first considered by  Constantino and Thurston \cite{ConT08}.
 The class is ``universal" in the sense that every orientable $3$-manifold with empty or toroidal boundary is obtained by  Dehn filling from the complement of a  fundamental shadow link.  Recently, Belletti \cite{Bel20} proved Conjecture \ref{VolC}
 for additional families of hyperbolic links in connected sums of  copies of $S^2\times S^1$. In \cite{Rol08}, Roland van der Veen introduced a family of links in $S^3$ known as the Whitehead chains, and he showed a version of the colored Jones polynomial volume conjecture is true for certain subfamilies.   K. H. Wong  \cite{Won19} proved Conjecture \ref{VolC} for  the same families of Whitehead chains and for  twisted Whitehead links.

 \subsection{Constructions of links}\label{1.1}
In this paper, we use purely topological methods to construct new infinite families of links in $S^3$ that satisfy Conjecture \ref{VolC}, and we show that several low crossing links  categorized by Livingston and Moore in \textit{LinkInfo} \cite{LinkInfo}  also satisfy the conjecture.  
To state our first result, let  $v_8 = 3.66386\dots$ denote the volume of a regular ideal hyperbolic octahedron.

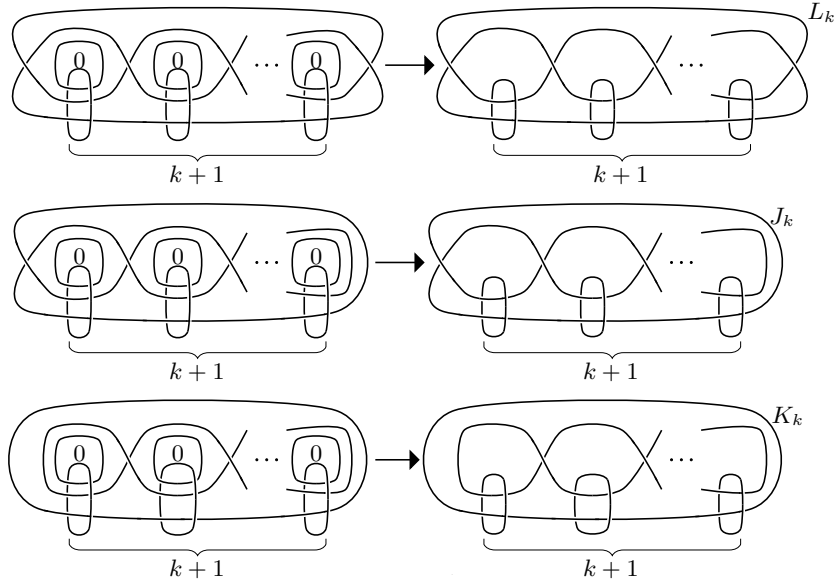
\begin{figure}
\labellist
\pinlabel \small{$0$} at 47 61 
\pinlabel \small{$0$} at 96.5 61
\pinlabel \small{$0$} at 166 61
\pinlabel \small{$0$} at 47 160
\pinlabel \small{$0$} at 96.5 160
\pinlabel \small{$0$} at 166 160
\pinlabel \small{$0$} at 47 259
\pinlabel \small{$0$} at 96.5 259
\pinlabel \small{$0$} at 166 259
\pinlabel \small{$k+1$} at 106 3
\pinlabel \small{$k+1$} at 314 3
\pinlabel \small{$k+1$} at 106 102
\pinlabel \small{$k+1$} at 314 102
\pinlabel \small{$k+1$} at 106 201
\pinlabel \small{$k+1$} at 319 201
\pinlabel \small{$L_k$} at 420 283
\pinlabel \small{$J_k$} at 400 180
\pinlabel \small{$K_k$} at 403 80
\endlabellist
\centering
\includegraphics[scale=.75]{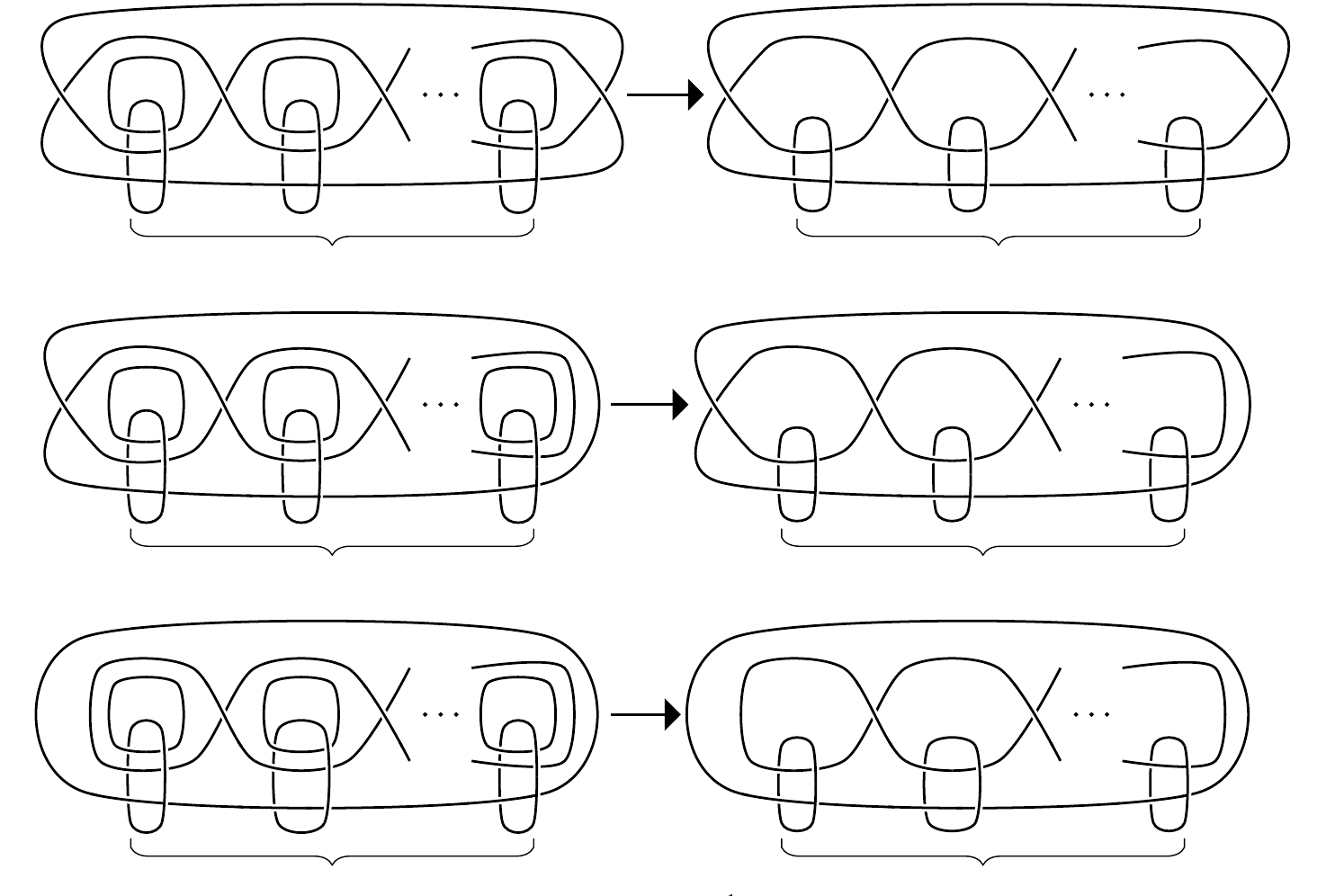}
\caption{On the left are the complexity $k$ fundamental shadow links embedded in connected sums of $S^2 \times S^1$, and on the right are their associated links in $S^3$. For the link family $\{L_k\}$ with $(5k+6)$ crossings, the number of components is $(k+2)$ for $k$ odd and $(k+3)$ for $k$ even. For the link family $\{J_k\}$ and $\{K_k\}$ with $(5k+5)$ and $(5k+4)$ crossings, the number of components is $(k+2)$ and $(k+3)$, respectively.}
\label{Fig:Families}
\end{figure}

\begin{customthm}{\ref{Families}}
Given  an integer $k \geq 1$, let $L_k$, $J_k$, and $K_k$ be the links in $S^3$ shown in Figure \ref{Fig:Families}. The complement of each $L_k$, $J_k$, and $K_k$ is hyperbolic
with volume $2kv_8$ and satisfies Conjecture \ref{VolC}.
 \end{customthm}

The link families in Theorem \ref{Families} are a generalization of the Borromean link and the Borromean twisted sisters in the sense that when $k=1$, their complements are homeomorphic with the links $L_1$, $J_1$, and $K_1$. See Corollary \ref{Borromean} and Figure \ref{Fig:TwistedSis} for more details.

To state our next result, we will use $B_n$ to denote the braid group of $n$ strings with standard generators $\sigma_1, \ldots, \sigma_{n-1}$.  
Recall that for a braid $b\in B_n$,  the  \emph{braided link} $\overline{b}$ is the link in $S^3$ consisting  of the  closure of  $b$ together with its braid axis. See Figure \ref{Fig:Braid}. 
For $k=1, 2, \ldots$, we define the braid $b_k$ as follows:

\begin{itemize}
\item For $k=1$, define $b_k:=\sigma_1^{-2} \sigma_2^{2} \in B_3$.
\item For $k=2m$, define $b_k\in  B_{k+3}$ by
 $$ b_k:=\prod_{i=1}^{m} (\sigma_{2i+1}\sigma_{2i} \cdots \sigma_3 \sigma_2 \sigma_1^2  \sigma_2^{-1} \sigma_3 \sigma_4 \cdots \sigma_{2i}
 \sigma_{2i+1}) (\sigma_{2i+2} \sigma_{2i+1} \cdots \sigma_3 \sigma_2^2 \sigma_3 \cdots \sigma_{2i+1} \sigma_{2i+2} ).$$
\item For $k=2m-1$, define $b_k\in  B_{k+3}$  by
$$b_k:=b\ \prod_{i=2}^{m} (\sigma_{2i} \sigma_{2i-1} \cdots \sigma_3 \sigma_2^2 \sigma_3 \cdots \sigma_{2i-1} \sigma_{2i}) (\sigma_{2i+1} \sigma_{2i} \cdots \sigma_3 \sigma_2 \sigma_1^2 \sigma_2^{-1} \sigma_3 \sigma_4 \cdots \sigma_{2i} \sigma_{2i+1}), $$
where $b=(\sigma_1)(\sigma_3 \sigma_2 \sigma_1^2 \sigma_2^{-1} \sigma_3).$ 
\end{itemize}

\begin{customthm}{\ref{Fibered}}
For $k\geq 1$,  let $M_k:=S^3\setminus \overline{b_k}$ denote the complement of the braided link $\overline{b_k}.$ Then
\begin{enumerate}  
\item $M_k$ fibers over  $S^1$  with fiber surface a disk with punctures,
\item $M_k$ is hyperbolic with volume $2kv_8,$ and
\item $M_k$ satisfies  Conjecture \ref{VolC}.
\end{enumerate}
\end{customthm}

As mentioned earlier  before the results of this paper, Conjecture \ref{VolC} was known for a family of  links  $\mathcal{W}$ in $S^3$ which are a subset of the  Whitehead chains containing one belt component  \cite{Won19}. Like the links of Theorems \ref{Families} and \ref{Fibered}, the links in 
$\mathcal{W}$  are octahedral with volumes  multiples of $v_8$. However, as the next results shows, our links are in general  distinct from these of
$\mathcal{W}$.

\begin{customprop}{\ref{Whitehead}}
For an integer $k\geq4$, no link of  Theorems \ref{Families} and  \ref{Fibered} has complements   homeomorphic to the complement of an element of $\mathcal{W}$. 
\end{customprop}

We will see each  of the complements of the link families appearing  in Theorems \ref{Families} and \ref{Fibered} can be obtained, up to homeomorphism, as the complement of a link in connected sums of $S^2\times S^1$ represented by a surgery presentation of a link  in $S^3$ of the form shown in the left diagrams of Figure \ref{Fig:Families} for Theorem \ref{Families}. Although one may be able to show the results by using Kirby calculus, a more natural method of proof  is to use techniques from
Turaev's shadow theory  to $3$-manifolds \cite{Tur94} to construct links with complements homeomorphic to complements of fundamental shadow links. This is due to the relationship between fundamental shadow links and a particular type of $2$-dimensional polyhedron equipped with a collapsing map. It will allow us to generalize our argument, showing that all links are sublinks of a fundamental shadow link in $S^3$ which we will discuss in Section \ref{1.2}.

 The general method of proof for Theorems \ref{Families} and \ref{Fibered} is to find certain fundamental shadow links  embedded into connected sums of $S^2\times S^1$ with  complements  homeomorphic to the complements  of links in $S^3$. Since these complements are homeomorphic, their Turaev-Viro invariants will be equal. It  was also shown in \cite{ConT08} that complements of fundamental shadow links have explicit decompositions into ideal hyperbolic octahedra and are therefore hyperbolic. Since the complements are hyperbolic,  the hyperbolic volume will be a topological invariant by Mostow-Prasad rigidity and will also be equal. This implies that these links in $S^3$ will also satisfy Conjecture \ref{VolC}. This gives a  topological proof of links in $S^3$ satisfying the  conjecture where the analysis was previously  done for the fundamental shadow links in connected sums of $S^2 \times S^1$ in \cite{BelDKY}. As stated before, we  show that the families of links we construct are distinct from the links in $S^3$ for which Conjecture \ref{VolC} was previously verified.

To give more detail, we can build a specific type of $2$-dimensional polyhedron introduced by Turaev \cite{Tur94} known as a \emph{shadow} with additional topological data known as the \emph{gleam}. From the shadow and gleam, we can construct a  connected and orientable $3$-manifold. In \cite{ConT08}, they introduce a particular family of shadows whose corresponding $3$-manifolds are the complements  of fundamental shadow links.
Starting with  specific shadows of fundamental shadow link complements, 
we will construct a new shadow with a link complement in $S^3$ as its corresponding $3$-manifold. 

In \cite{ConT08}, the authors show there is a strong relationship between a shadow and a $2$-dimensional polyhedron known as the \emph{Stein surface}. In particular when the manifold is a link in $S^3$, we will see  that Stein surfaces and shadows are equivalent where the Stein surface is equipped with a  map $h$ from the manifold to the Stein surface. Through the map $h$, we can drill out tori from our link complement in $S^3$ to obtain a $3$-manifold homeomorphic to the complement of the fundamental shadow link. Since the manifolds are homeomorphic, they will both satisfy Conjecture \ref{VolC}.

Our methods also allow us to verify Conjecture \ref{VolC} for several octahedral links up to $11$-crossings.  There are only finitely many fundamental shadow links for a fixed volume $2kv_8$, therefore we can check possible homeomorphisms between link complements computationally.   Using \textit{SnapPy} \cite{SnapPy}, we can verify that all but one of the  links of volume  $2v_8$ up to 11-crossings given in
 \textit{LinkInfo} \cite{LinkInfo}   have complements homeomorphic to the complement of the simplest fundamental shadow links. See Table \ref{Table:Links}  and Figure \ref{Fig:SixEx} for details.  With the notation of  \textit{LinkInfo}, \textit{SnapPy} determines that the link $L_{10n59}$ is not homeomorphic to a fundamental shadow link. This could be the result of a numerical error and is not rigorous, therefore the link $L_{10n59}$ may still be a   fundamental shadow link. We believe this is not the case, and that the link complement is  not homeomorphic to the  complement of a fundamental shadow link; however,  to show this is true would require a different technique.

\begin{thm}\label{Links} Using the notation of  \textit{LinkInfo},
each of the following links is hyperbolic with volume $2v_8$ and satisfies Conjecture \ref{VolC}.
 \begin{table}[h]
\centering
\begin{tabular}{|c| c| c| c| c| c|}
\hline
 $L_{6a4}$ & $L_{8n5}$ & $L_{8n7}$ & $L_{9n25}$ & $L_{9n26}$ & $L_{10n32} $ \\
\hline
 $L_{10n36}$ & $L_{10n70}$ & $L_{10n84}$ & $L_{10n87} $ & $L_{10n97} $ & $L_{10n105}$ \\
\hline
 $L_{10n108}$ & $L_{11n287}$ & $L_{11n376}$& $L_{11n378} $  & $ L_{11n385}$ &  \\
 \hline
\end{tabular}
\label{Table:Links}
\end{table}
\end{thm}

\subsection{All links are sublinks of fundamental shadow links}\label{1.2}

From using  methods similar to those in Section \ref{1.1}, we are able to show that all links are contained in fundamental shadow links in $S^3$; thus, all links can be augmented to one that satisfies Conjecture \ref{VolC}. The process for augmenting a link is shown in Section \ref{4}. Our result is stated in terms of closed braids; however, by a classical theorem of Alexander \cite{Ale23}, every link can be obtained as the closure of a braid. We begin by defining the \textit{length} of a braid $b$ to be the number of standard generators appearing in the word, and we denote the braid closure by $\hat{b}$. The result requires each generator from the braid group to appear; however, since the algorithm will work for  unreduced words in $B_n$, it works for every link. For example, the unlink of two components can be viewed as the braid closure of $\sigma_1 \sigma_1^{-1} \in B_2$.

\begin{customthm}{\ref{Sublink}}
Let $b\in B_n$ be an unreduced word of length $k$ with each generator appearing at least once. Then $\hat{b}$ is a sublink of a  link  $L$ in $S^3$  where the complement $S^3 \backslash L$ satisfies Conjecture \ref{VolC} and has volume $2kv_8$.
\end{customthm}

Theorem \ref{Sublink} should be compared to a result of Baker \cite{Bak02} and  Roland van der Veen \cite{Rol09}  that states all links are sublinks  of arithmetic links. In particular, Theorem \ref{Sublink} has similar properties to Corollary $1$  in \cite{Rol09}. By using augmented knotted trivalent graphs, Roland van der Veen showed that any link is a sublink of a link that satisfies an $SO(3)$ version of the original volume conjecture.

\subsection{Applications to quantum representations} \label{1.3}
Our results on Conjecture \ref{VolC} also have applications to a conjecture of  Andersen, Masbaum, and Ueno (AMU Conjecture) about quantum representations of surface mapping class groups.
For a compact and oriented surface $\Sigma_{g,n}$ with genus $g$ and $n$ boundary components, let $\mathrm{Mod}(\Sigma_{g,n})$ be its mapping class group fixing the boundaries, and let $|\partial \Sigma_{g,n} |$ denote the set of boundary components of $\Sigma_{g,n}$. Now define $U_r$ to be the set of nonnegative and even integers less than $r-2$.  Given an odd integer $r\geq 3$,  a primitive $2r$-th root of unity, and a coloring $c: |\partial \Sigma_{g,n} | \rightarrow U_r$,
the $\mathrm{SO}(3)$-Witten-Reshetikhin-Turaev TQFT \cite{ BlaHMV95, ResT91, Tur94} gives a
projective representation $$\rho_{r,c} : \mathrm{Mod}(\Sigma_{g,n}) \rightarrow \mathrm{PGL}_{d_{r,c}}(\mathbb{C})$$
called the $\mathrm{SO}(3)$-\textit{quantum representation} where $d_{r,c}$ is an integer dependent on the level $r$ and coloring $c$. 
The AMU conjecture
 relates the Nielsen-Thurston classification of elements of the mapping class group for a compact surface $\Sigma_{g,n}$ to its quantum representations.

\begin{conjecture}[\cite{AndMU}, AMU conjecture]\label{AMU}
An element of the mapping class  $f \in Mod(\Sigma_{g,n})$ has pseudo-Anosov parts if and only if, for any big enough level $r$, there is a choice of coloring $c$ of the elements of $\partial \Sigma_{g,n}$ such that $\rho_{r,c}(f)$ has infinite order. 
\end{conjecture} 
In \cite{AndMU}, the authors also verified Conjecture \ref{AMU} for $\Sigma_{0,4}$ and later  Santharoubane for $\Sigma_{1,1}$ in \cite{San12}.  Egsgaard and Jorgensen in \cite{EgsJ16} and Santharoubane in \cite{San17}, by using different approaches, produced elements of $Mod(\Sigma_{0,2n})$  satisfying  Conjecture \ref{AMU}. For higher genus surfaces, there have been results using different methods. In \cite{MarS16}, March\'e and Santharoubane construct  finitely many conjugacy classes of pseudo-Anosov elements for $Mod(\Sigma_{g,1})$, and  Detcherry and Kalfagianni construct cosets of $Mod(\Sigma_{g,1})$ for sufficiently large $g$ that satisfy  Conjecture \ref{AMU} in \cite{DetK202}. Now we remark that two  elements $f, g \in Mod(\Sigma_{g,n})$  are \emph{independent} if there is no $h \in Mod(\Sigma_{g,n})$ such that both $f$ and $g$ are  conjugate to non-trivial powers of $h$. As explained in \cite{DetK202}, the motivation behind the definition of independent is if $f$ and $g$ are not independent, then $f$ satisfies the AMU conjecture if and only if $g$ does. 
  It has  been shown by Detcherry and Kalfagianni in \cite{DetK19} that there are infinitely many pairwise independent pseudo-Anosov elements in $Mod(\Sigma_{g,2})$ for $ g \geq 3$ and for $Mod(\Sigma_{g,n})$ where $g \geq n \geq 3$. 

Using the approach of \cite{DetK19} and by identifying 3-manifolds with explicit fibrations among the link complements in Theorems \ref{Families} and \ref{Links}, we obtain the following.   For more details, see Section \ref{5}.

 \begin{customthm}{\ref{4Sphere} }
Given $g \geq 0$, there is a pseudo-Anosov element of $Mod(\Sigma_{g,4})$ that satisfies Conjecture \ref{AMU}. Furthermore for $g \geq 3$, there are infinitely many pairwise independent pseudo-Anosov elements in $Mod(\Sigma_{g,4})$ that satisfy Conjecture \ref{AMU}.
\end{customthm}
 
As a byproduct of the proof of Theorem  \ref{4Sphere}, we also extend the known results of Conjecture \ref{AMU} to produce examples in  $Mod(\Sigma_{n-1,n})$ and $Mod(\Sigma_{n-2,n})$ for $n \geq 5$, as well as in $Mod(\Sigma_{1,3})$ and $Mod(\Sigma_{3,3})$. 

Finally as an application of Theorems \ref{Fibered} and \ref{Sublink}, we also produce 
explicit elements in $Mod(\Sigma_{0,n})$, for all $n>3$,  that satisfy Conjecture \ref{AMU}. In particular for most choices of $n$, the elements we find are pseudo-Anosov  as stated in the following corollary. 

\begin{cor}
For $n \geq 4$, if either $n \geq 9$ or $n$ is even, then there exists a pseudo-Anosov  element of $Mod(\Sigma_{0,n})$ that satisfies Conjecture \ref{AMU}. 
\end{cor}

 To the author's knowledge, this provides the first examples of pseudo-Anosov elements satisfying Conjecture \ref{AMU} in $Mod(\Sigma_{0,n})$ for an odd $n$. See Corollary \ref{Even} for the explicit elements  as well as more details.

The paper is organized as follows: We define the previously mentioned shadows, fundamental shadow links, and the Stein surfaces in Section \ref{2}. We also describe the relationship between the shadows and the Stein surfaces for links in $S^3$. In Section \ref{3}, we prove the results of Theorems \ref{Families} and \ref{Fibered} by constructing link complements from particular shadows, and we describe the implementation of \textit{SnapPy} to prove Theorem \ref{Links}. In Section \ref{4}, we introduce and prove an algorithm for augmenting a link in $S^3$ into a fundamental shadow link.   In Section \ref{5}, we  extend the results of Conjecture \ref{AMU} to the previously mentioned elements of the mapping class groups by using the links from Section \ref{3} along with a result by Detcherry and Kalfagianni which shows Conjecture \ref{VolC} implies Conjecture \ref{AMU} \cite{DetK19}. 

\begin{ack*}
\upshape{The author wishes to thank his advisor Efstratia Kalfagianni for helpful discussions and for her guidance. Furthermore, the author would  like to thank Renaud Detcherry for his comments and suggestions on an earlier version of this paper. The author would also like to thank the referee for carefully reading and providing suggestions to improve the overall exposition of the paper. This material is based on  research partially supported  by NSF grants  DMS-1404754, DMS-1708249, DMS-2004155  and by a Dr. Paul and Wilma Dressel Endowed Scholarship from the Department of Mathematics at Michigan State University.  }
\end{ack*}


\section{Preliminaries}\label{2}

\subsection{Shadows of 3-manifolds}\label{2.1}
We will start with a summary of  shadows of $3$-manifolds. As noted previously, shadows allow for an approach for studying  $3$-manifolds in terms of $2$-dimensional polyhedra.  A more detailed discussion by Turaev can be found in \cite{Tur94}.

We will begin with shadows of oriented, closed, and connected $3$-manifolds.

\begin{figure}
\labellist
\pinlabel \small{Region} at 28 15
\pinlabel \small{Edge} at 170 15
\pinlabel \small{Vertex} at 312 0
\endlabellist
\centering
\includegraphics[scale=1]{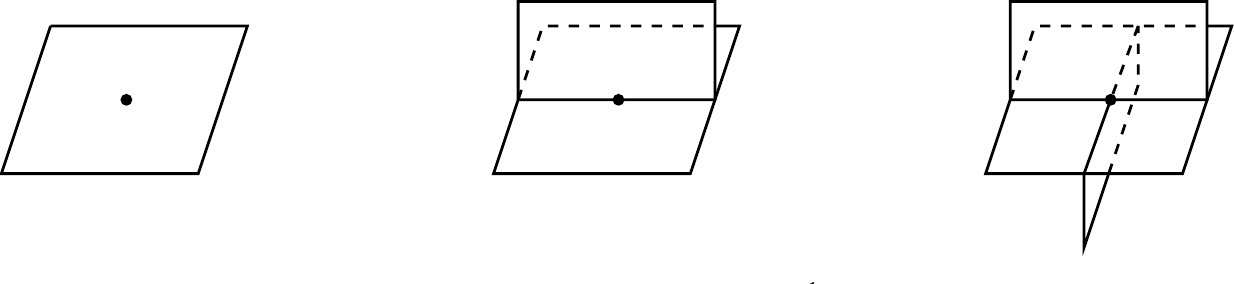}
\caption{Local models of the simple polyhedron.}
\label{Fig:Local}
\end{figure}

\begin{definition}\label{Def:Shadows}
A simple polyhedron $P$ is a compact topological space which locally is homeomorphic to one of the models in Figure \ref{Fig:Local}. The set of points with local model of the right two types form a $4$-valent graph called the $\emph{singular set}$ denoted by $Sing(P)$. In particular, the set of points locally homeomorphic to the third model are known as $\emph{vertices}$, and the set of points locally homeomorphic to the second model are known as \emph{edges}. The set of connected components $P \backslash Sing(P)$ are called $\emph{regions}$ of $P$, and the number of vertices of $P$ is called the $\emph{complexity}$ of $P$. The set of points with local models corresponding to the boundaries of  Figure \ref{Fig:Local} are known as the $\emph{boundary}$ of $P$, denoted $\partial P$. If $P$ has no boundary components, then we say $P$ is $\emph{closed}$. 
\end{definition}

\begin{definition}\label{Def:4Mfd}
Let $W$ be a piecewise linear, compact, and oriented $4$-manifold, and let $P$ be a closed simple polyhedron. We say that $P \subset W$ is a $\emph{shadow}$ of $W$  if $W$ collapses onto $P$ and if $P$ is locally flat; otherwise stated, for each point $p \in P$, there exists a neighborhood $Nbd(p; P)$ that is contained in a $3$-dimensional submanifold of $W$. 
\end{definition}

\begin{definition}\label{Def:3Mfd}
We say that a closed simple polyhedron $P$ is a $\emph{shadow}$ of an oriented $3$-manifold $M$ if $P$ is a shadow of a $4$-manifold $W$ with $\partial W = M.$
\end{definition}

\begin{thm}[\cite{Tur94}, Theorem IX.$2.1.1$]\label{ManiShad}
 Every closed, oriented, and connected $3$-manifold has a shadow.
\end{thm}

From  Theorem \ref{ManiShad}, every closed, oriented, and connected $3$-manifold has a shadow; however, a simple polyhedron $P$ can be a shadow for different non-homeomorphic manifolds. In order to distinguish the two manifolds, we need the additional topological information known as gleams. We will not go into detail on gleams; however, again one can refer to \cite{Tur94} for a more in-depth discussion. In particular, we will only define the  $\mathbb{Z}_2$-gleam of a region in its simplest case when the closure of the region is embedded in the simple polyhedron, although  it can also be defined otherwise.

\begin{definition}\label{Def:HalfGleam}
Let $D$ be an embedded region in $P$ such that the closure $\overline{D}$ is also embedded with $\partial \overline{D} \subset Sing(P)$.  Now let $N(D)$ be a regular neighborhood of $\partial \overline{D}$ in $P\backslash Int(D)$. If $N(D)$ collapses over a M\"{o}bius strip, then the \emph{$\mathbb{Z}_2$-gleam} of the region $D$ is $1$, otherwise it is $0$. 
\end{definition}

\begin{definition}\label{Def:Gleam}
A $\emph{gleam}$ is a coloring of all regions by $\frac{1}{2} \mathbb{Z}$ where the coloring of a region is an integer if and only if the $\mathbb{Z}_2$-gleam is $0$. A simple polyhedron equipped with a gleam is called a $\emph{shadowed polyhedron}$. 
\end{definition}

The general idea is that thickening our polyhedron to a $4$-manifold  will require additional information on how to glue together  components. The data from the gleams serve this purpose.

 Now in order to consider shadows for links embedded into a $3$-manifold $M$, where $M$ has possible toroidal boundary components, we need to extend our definition by including a coloring of the boundary for non-closed simple polyhedron.

\begin{definition}\label{Def:ColoredBoundary}
A $\emph{boundary-decorated simple polyhedron}$ $P$ is a simple polyhedron where each boundary component of $P$ is labeled as either $\emph{internal}$, $\emph{external}$, or $\emph{false}$. Let us label these subsets as $\partial_{int}(P)$, $\partial_{ext} (P)$, and $\partial_f (P)$, respectively. 
\end{definition}

The boundary-decorated simple polyhedron can be shadows for a pair $(M,L)$ where $M$ is an oriented $3$-manifold with link $L$ embedded in $M$.

\begin{definition}\label{Def:BoundaryShadow}
A boundary-decorated simple polyhedron $P$ is a $\emph{shadow}$ for the pair $(M,L)$ if $W$ collapses onto $P$, $P$ is locally flat, and $(M,L) = (\partial W \backslash IntNbd(\partial_{ext} (P), \partial W) , \partial_{int} (P))$.
\end{definition}

The intuition being that the internal boundaries are the link, the external boundaries correspond to drilling out tori, and the false boundaries are ignored. Similarly to before, we can define a gleam for a boundary-decorated simple polyhedron.

\begin{definition}\label{Def:BoundaryGleam}
A coloring of the regions by half-integers of a boundary-decorated simple polyhedron that does not touch $\partial_f(P) \cup \partial_{ext}(P)$ is called a $\emph{gleam}$. A boundary-decorated simple polyhedron equipped with a gleam is called a \emph{boundary-decorated shadowed polyhedron}. 
\end{definition}

In \cite{Tur94}, Turaev showed that if a simple polyhedron embedded in a compact oriented smooth $4$-manifold $W$ is a shadow of $\partial W$, then there exists an induced gleam of the polyhedron. Conversely, he showed given a  boundary-decorated shadowed polyhedron, there exists a construction of a compact and oriented $4$-manifold $W$ with an embedding $P$ into $W$ where the gleam of $P$ coincides with the induced gleam given by the embedding in $W$. This implies that one way we can study $3$-manifolds is through the use of boundary-decorated shadowed polyhedron. 
We will show a general sketch of Turaev's proof to illustrate how to pass from a boundary-decorated shadowed polyhedron to a $3$-manifold which can be found in \cite{Tur94}.

\begin{thm}[\cite{Tur94}, Section IX.$6.1$]\label{Rec}
From a boundary-decorated shadowed polyhedron $P$, we can construct a pair $(W_P, P)$ where $W_P$ is a smooth, compact, and oriented $4$-manifold with $P$ embedded such that $P$ is a boundary-decorated shadowed polyhedron of the $3$-manifold $\partial W_P$. In addition, the regions are colored such that the induced gleam coincides with the gleams of the  boundary-decorated shadowed polyhedron. 
\end{thm}

\begin{proof}
We will begin with the case when $P$ is a closed simple polyhedron with non-empty connected singular set $Sing(P)$. Notice a neighborhood $S(P)= Nbd(Sing(P);P)$ can be obtained as a series of compositions of the two rightmost local models in Figure \ref{Fig:Local}. We can thicken these components to obtain a $3$-manifold $N$ with $S(P)$ properly embedded in $N$ such that  the gluing of the components is induced from the original simple polyhedron $P$. Now  consider the thickening to a $4$-manifold  $W_0$ that is the $[-1,1]$ bundle over $N$. Note that this $W_0$ is  orientable. 

Now we consider the remaining components of the shadow $P$ which are regions such as the  leftmost local model in Figure \ref{Fig:Local}. Let $R$ be such a component, and let $W_R$ be the thickening to a $4$-manifold in the same way as previously. We now need to consider how the boundaries of $W_R$ and $W_0$ glue to each other. Notice the simple closed curves coming from the intersection of $R \cap \partial (S(P))$. This will be the core of an annulus or a M\"{o}bius band in both the $3$-dimensional thickening of $R$ and $S(P)$. The initial gleam of the shadowed polyhedron $P$ is used as the framing of the core  to identify the components of the annuli and M\"{o}bius bands.

Now we consider the case when $P$ has boundary. For our purposes, the boundaries are copies of $S^1$. Let $\{l_i\}$ be the collection of components of the boundary where $l$ is their union. If we thicken $l$ to the framed link $l \times [-1,1]$ such that $l \times \{0\} = \partial(P)$,  then we obtain a new simple polyhedron $\overline{P}$ where $Sing(\overline{P}) = Sing(P) \cup \partial(P)$. Now by following the same construction as before with $\overline{P}$, we obtain a $4$-manifold $W_{\overline{P}}$ such that $l \times [-1,1] \subset \partial(W_{\overline{P}})$.
\end{proof}


\subsection{Fundamental shadow links}\label{2.2}

The fundamental shadow links are a family of links in connected sums of $S^2 \times S^1$ with well understood hyperbolic volumes. These manifolds are introduced and discussed in-depth by Constantino and Thurston in \cite[Section $3.6$]{ConT08}. We will be using Theorem \ref{Rec} on a particular family of boundary-decorated shadowed polyhedron to obtain these manifolds. As noted in the introduction, these links are ``universal" in the sense that every orientable $3$-manifold with empty or toroidal boundary can be obtained as the complement of a Dehn filling of a fundamental shadow link along some subset of the components of the link.  Again, Conjecture \ref{VolC} for this family of links was proven in \cite{BelDKY}. This was done by utilizing the explicit hyperbolic volume of this family, as well as the relatively simple form of the Turaev-Viro invariant.

We will begin by starting with a closed simple polyhedron $P$ with at least one vertex such that $Sing(P)$ is connected. Now let $S(P)$ be a regular neighborhood of $Sing(P)$ in $P$ such that $S(P)$ is a composition of the  two rightmost local models of Figure \ref{Fig:Local}, and define $P'$ to be the boundary-decorated shadowed polyhedron with $2$-dimensional simple polyhedron $S(P)$ and each boundary component colored as  external. $P'$ is a boundary-decorated shadowed polyhedron of a $3$-manifold $M_{P'}$ with boundary where $\pi: M_{P'} \to P'$ is the collapsing map. We can understand this manifold by looking at the preimage of $P'$ under the collapsing map.  The preimage of the second local model will be a product of a pair of pants with $[-1,1]$, and the preimage of the third local model is a genus $3$ handlebody. These pieces meet along pairs of pants, and the resultant manifold $M_{P'}$ has boundary a union of tori. From \cite{ConT08}, we have the following.

\begin{prop}[\cite{ConT08}, Proposition 3.33]\label{FSLVol}
The manifold $M_{P'}$ can be equipped with a complete hyperbolic metric with volume $2k v_8$ where $k$ is the number of vertices of ${P'}$ and $v_8$ is the volume of a regular ideal hyperbolic octahedron. 
\end{prop}

Moreover, it can be shown that the manifold $M_{P'}$ can be realized as the complement of a link embedded in a closed, connected, and oriented $3$-manifold. This is constructed from the same simple polyhedron $S(P)$ with different boundary colorings.

\begin{definition}\label{Def:FSL}
Let $P''$ be the complexity $k$ boundary-decorated shadowed polyhedron with $2$-dimensional simple polyhedron $S(P)$ and with boundary components colored as internal. The $4$-manifold $W_{P''}$ obtained from Theorem \ref{Rec} of the boundary-decorated shadowed polyhedron $P''$ collapses onto the  pair
$$(M_{P''}, L_{P''}) = (\partial W_{P''} , \partial_{int} (P'') ) = (\#^{k+1} (S^2 \times S^1), L_{P''}). $$ 
The pair of this link and $3$-manifold is known as the \emph{fundamental shadow links}. Furthermore, the $3$-manifold $M_{P'}$ and the link complement of $L_{P''}$ in the manifold $M_{P''}$ are homeomorphic. 
\end{definition}

\begin{figure}
\labellist
\pinlabel \small{Vertex} at 27 5
\endlabellist
\centering
\includegraphics{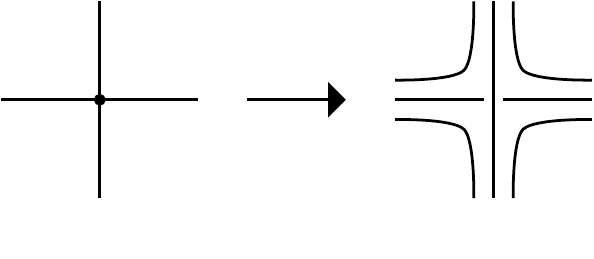}
\caption{Replacing vertices of the graph.}
\label{Fig:Graph}
\end{figure}

Although these fundamental shadow links are embedded in connected sums of $S^2 \times S^1$, they have a description in terms of $0$-surgeries on a set of $(k+1)$  unknotted link components of a link in $S^3$ as noted in Proposition $3.35$ \cite{ConT08}. This can be done through the following procedure:

\begin{itemize}
\item Immerse the simple polyhedron $S(P)$ into $S^3$.
\item Consider a maximal tree of the singular set $Sing(P)$.
\item Encircle with $0$-framed unknotted components all triples of strands from $\partial (S(P))$ running over edges not belonging to the tree $T$.
\item The union of the $0$-framed unknotted components and the strands from $\partial (S(P))$ form a surgery presentation in $S^3$ of a fundamental shadow link.
\end{itemize}
 
 Since the singular set $Sing(P)$ is a $4$-valent graph, we can also construct fundamental shadow links  as follows where the \emph{complexity} of a graph is the number of vertices:
 
\begin{itemize}
 \item We consider $G$ to be a planar $4$-valent graph of complexity $k$ embedded in $S^3$.
 \item As before, we let $T$ be a maximal tree of the graph $G$.
 \item Replace each edge by any $3$-braid, replace  each vertex by the $6$ strands  in Figure \ref{Fig:Graph} or its mirror, and encircle each braid corresponding to an edge $G\backslash T$ by a $0$-framed unknotted link.
 \item The union of the strands from the $3$-braid and the $0$-framed unknotted link again give us a surgery presentation in $S^3$.  
\end{itemize}

For a given planar $4$-valent graph $G$, there is not a unique associated fundamental shadow link. Depending on which $3$-braids you replace each edge with, we may end with differing non-homeomorphic fundamental shadow link complements.

\begin{remark}\label{Rem:Immersed}
In constructing a fundamental shadow link from a planar $4$-valent  graph, if we exclude the $0$-framed unknotted components, then there exists a simple polyhedron immersed in $S^3$ with boundary the  components of our link. This can be viewed as the immersed simple polyhedron $S(P)$ obtained from the first surgery presentation construction of the shadow of a fundamental shadow link.
\end{remark}


\subsection{Stein factorization of stable maps of links}\label{2.3}
In \cite{ConT08}, Constantino and Thurston show that given a $3$-manifold $M$ and a map $f:M \to \mathbb{R}^2$, one can recover a boundary-decorated shadowed polyhedron for $M$ through the Stein factorization of the map $f$. We will introduce this factorization and later see how it relates to links embedded in $S^3$. 

As a brief summary, the Stein factorization of a stable map $f=g \circ h$ decomposes $f$ into a map $h$ of connected fibers and a map $g$ which is finite-to-one. Roughly speaking, $f$ can be decomposed into maps $g: W_f \to \mathbb{R}^2$ and $h: M \to W_f$ where $W_f$ is a $2$-dimensional polyhedron. In our case, $W_f$ will be a simple polyhedron where the boundary, edges, and vertices of the simple polyhedron correspond to the singularities of the map $f$. At a generic point of  $W_f$, the preimage under the map $h$  is a circle bundle inside a $3$-manifold which is analogous to the $3$-thickening of a shadow. Along these circle bundles at the generic points, we can glue disks which is similar to the $4$-thickening process.  This results in a   $4$-manifold which collapses onto the simple polyhedron along the disk with boundary the $3$-manifold $M$. This is close to a boundary-decorated shadowed polyhedron of a manifold, the difference being that the topological data of the gleam is not explicitly stated on the simple polyhedron. 
We will now define the Stein surface with more information on these definitions being discussed by Levine in \cite{Lev85} and Saeki in \cite{Sae96}.

\begin{definition}\label{Def:StableMap}
Let $M$ be a closed $3$-manifold and $f$ a smooth map from $M$ to an orientable $2$-manifold $\Sigma$. Suppose $f$ also satisfies the following conditions.
For each $p \in M$ and $f(p)$, there exist local coordinates centered at $p$ such that $f$ can be described in one of the following ways:
\begin{enumerate}[i).]
\item $(u,x,y) \mapsto  (u,x)$.
\item $(u, x, y) \mapsto (u, x^2 + y^2)$.
\item $(u, x, y) \mapsto (u, x^2 - y^2)$.
\item $(u, x, y) \mapsto (u, y^2 + ux - x^3)$.
\end{enumerate} 
We say $p$ is a \emph{regular point}, a \emph{definite fold point}, an \emph{indefinite fold point}, or a \emph{cusp point} for the cases $(i), (ii), (iii), (iv)$, respectively. In addition, if $S(f)$ are the singular points of $f$, then suppose $f$ also satisfies these global conditions:
\begin{enumerate}
\item[v).] $f^{-1} \circ f(p) \cap S(f) = \{ p \}$ for a cusp point.
\item[vi).] The restriction of $f$ to the singular points which are not cusp points is an immersion with normal crossings. 
\end{enumerate}
We say that such a map $f$ is a \emph{stable map}. 
\end{definition}

\begin{definition}\label{Def:SMap}
Given a compact and orientable $3$-manifold $M$ with (possibly empty) boundary of tori. A smooth map $f$ of $M$ into an orientable $2$-manifold $\Sigma$ is called an \emph{S-map} if the restriction to the interior of $M$ is a stable map, and  each $p \in \partial M$ can be described locally by the map $(u,x,y) \mapsto (u,x)$ where $x=0$ corresponds to points on the boundary. 
\end{definition}

\begin{definition}\label{Def:SteinSurface}
For a given $S$-map, consider the space $W_f$ which is the quotient of  the space $M$ along the fibers of $f$ with map $h : M \to W_f$. 
Define the map $g: W_f \to \Sigma$ such that $f = g\circ h$. This factorization of the map $f$ is known as the \emph{Stein factorization} with corresponding \emph{Stein surface} $W_f$.  
\end{definition}

We can further extend this definition to include links in a manifold $M$.

\begin{definition}\label{Def:LinkStableMap}
Let $M$ be a compact orientable $3$-manifold $M$ with (possibly empty) boundary of tori and a (possibly empty) link $L$. Let $f$ be an $S$-map of the manifold $M$ into an orientable $2$-manifold $\Sigma$. We say that $f$ is an $S$-map of the pair $(M,L)$ if the link $L$ is contained in the set of definite points of $f$. When $M$ is a closed $3$-manifold, the map $f$ is called a \emph{stable map} of the pair $(M,L)$. 
\end{definition}


\subsection{Stein surfaces and shadows of $3$-manifolds}\label{2.4}

As stated before, the Stein surfaces of a manifold $M$ given by a Stein factorization are closely related to the shadow of  $(M,L)$. Let us construct a boundary-decorated shadowed polyhedron for  the specific case $(S^3, L)$ where a similar example  can be found in Section $3.2$ of \cite{ConT08}. We will see that  the boundary-decorated shadowed polyhedron is equivalent to the Stein surface for the pair $(S^3,L)$.

 Let $D$ be a flat oriented disk in the $4$-ball $B^4$, and let $D'$ be the closure of $D$ minus a neighborhood of its boundary in $D$. Now let $\pi: S^3 \to D$ be the projection induced from collapsing $B^4$ onto the disk $D$. We  can assume the preimage of $D'$ under the map $\pi$ is a solid torus $T$, and we can define $T' = S^3 \backslash Int (T)$. Now given an $n$-component link $L$, we can isotope $L$ such that it is contained in $D' \times [-\epsilon, \epsilon]$ for sufficiently small number $\epsilon$ with $\pi |_T$ generic with respect to the link $L$. 

Now the projection $\pi(L)$ onto the closed disk $D'$ with crossing information gives a diagram $D_L$ of the link. Consider the mapping cylinder obtained from the quotient 

$$P'_L = (L \times [0,1] \sqcup D')/ ((x,0) \sim \pi(x))$$
where we decorate the components of $L \times \{1\}$ with the internal coloring. Now define $P_L$ to be the boundary-decorated shadowed polyhedron obtained from collapsing the region touching the remaining uncolored boundary component  onto the components of $\pi(L)$. There is a choice of gleam of the regions not touching $\partial_{int} (P_L)$ determined by the simple polyhedron of  $P_L$ such that it is a boundary-decorated shadowed polyhedron for the pair $(S^3,L)$. In this paper, the gleam of such a region $R$ is $\frac{V_R}{2}$ where $V_R$ is the number of vertices touching $R$. For more details, see Section $3.2$ in  \cite{ConT08}.

\begin{figure}
\centering
\includegraphics[scale=.9]{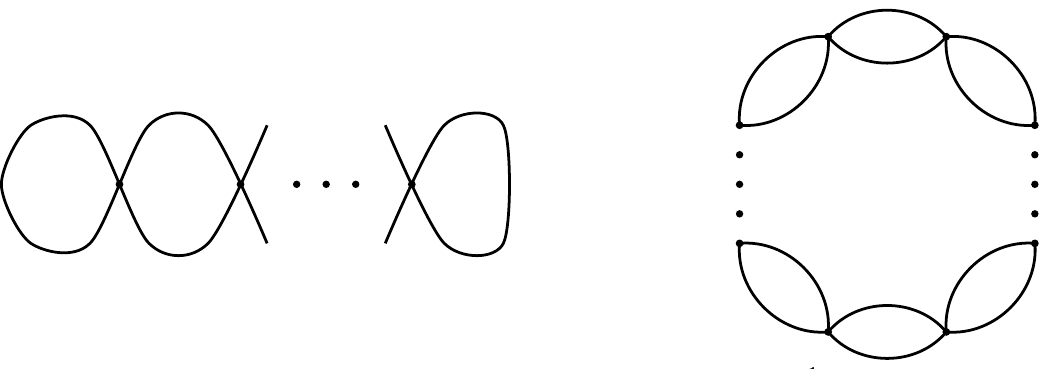}
\caption{Examples of planar $4$-valent  graphs of complexity $k$.}
\label{Fig:Planar}
\end{figure}

From \cite{ConT08}, given the  simple polyhedron of $P_L$, we can construct a stable map $f$ of the pair $(S^3, L)$ to $\mathbb{R}^2$ such that the Stein surface $W_f$ can be identified to the boundary-decorated shadowed polyhedron $P_L$ in a natural way with map $h: S^3 \to W_f$. We will describe the preimages of points on the Stein surface into $S^3$ with  details being discussed in the proof of Theorem $3.5$ in \cite{IshK17} by Ishikawa and Koda. Their argument is for links inside compact and orientable $3$-manifolds with (possibly empty) boundary consisting of tori; however in Section $4$ of \cite{IshK17}, the authors consider the specific case for a link in $S^3$. For our purposes, the most important points are those coming from $P_L \backslash (L\times [0,1])$.

\begin{prop}[\cite{IshK17}, Theorem 3.5]\label{Stein}
Let $L$ be a link in $S^3$. There exists a stable map $f: (S^3, L) \to \mathbb{R}^2$ with Stein surface $W_f$ and Stein factorization $f= g\circ h$ such that the preimage of a point in $W_f$ can be viewed as one of the five following:

\begin{enumerate}[i).]
\item If $x_1$ is a point on $P_L \backslash (L\times [0,1])$, then $h^{-1}(x_1) $ is the regular fiber $\{x_1\} \times S^1 \subset T$. 
\item If $x_2$ is a point on $L\times (0,1)$, then $h^{-1}(x_2)$ is a meridian of the corresponding link component of $L \times \{1\}$.  
\item If $x_3$ is a point on $L \times \{1\}$, then $h^{-1}(x_3)$ is a point on the corresponding link in $S^3$. 
\item Let $x_4$ be a point from the singular set $S(P)$ which are not vertices. Locally around $x_4$, there are regions separated by the singular set. There exists a triple of points, one from each region, such that their preimages are circles which are the boundary to a pair of pants with $h^{-1}(x_4)$ as its spine. 
\item Let $x_5$ be a point from the set of vertices of the singular set $S(P)$. As in case $iv)$, there exist preimages of points from the regions around the vertex which pullback to circles such that they bound a surface with spine $h^{-1}(x_5)$. 
\end{enumerate}
\end{prop}


\section{Families of links in $S^3$ satisfying  Conjecture \ref{VolC}}\label{3}

We will construct examples of links in $S^3$ with complements homeomorphic to the complements of fundamental shadow links. Since Conjecture \ref{VolC} has been proven for the latter, this gives a topological proof of   Conjecture \ref{VolC} for the links in $S^3$.

\begin{thm}\label{Families}
Given  an integer $k \geq 1$, let $L_k$, $J_k$, and $K_k$ be the links in $S^3$ shown in Figure \ref{Fig:Families}. The complement of each $L_k$, $J_k$, and $K_k$ is hyperbolic
with volume $2kv_8$ and satisfies Conjecture \ref{VolC}.
\end{thm}

In order to prove Theorem \ref{Families}, we will be using the following.

\begin{lem}[\cite{ConT08}, Lemma $3.25$]\label{Cap}
Let $P$ be a boundary-decorated shadowed polyhedron of $(M,L)$ where $L$ is a framed link and $M$ is an oriented $3$-manifold. Then the manifold $M'$ obtained from Dehn filling $L$ has a shadow $P'$ given by capping components of $\partial P$ by disks. 
\end{lem}

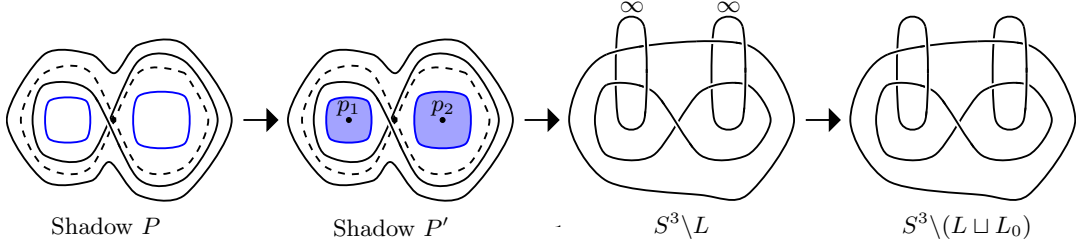
\begin{figure}
\labellist
\pinlabel \small{Shadow $P$} at 42 0
\pinlabel \small{Shadow $P'$} at 163 0
\pinlabel \small{$S^3 \backslash L$} at 286 0
\pinlabel \small{$S^3 \backslash (L \sqcup L_0$)} at 407 0
\pinlabel \small{$p_1$} at 145 50
\pinlabel \small{$p_2$} at 185 50
\pinlabel \small{$\infty$} at 265 93
\pinlabel \small{$\infty$} at 306 93
\endlabellist
\centering
\includegraphics[scale=.9]{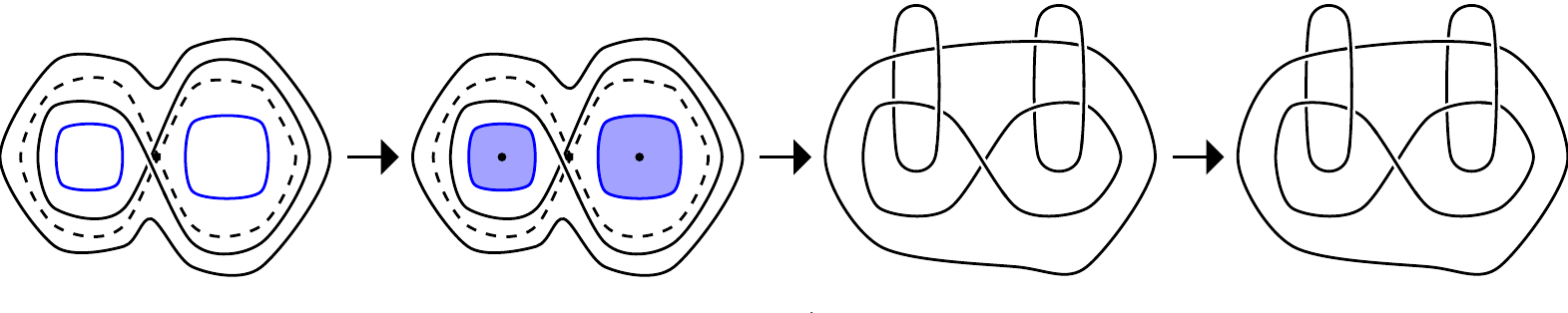}
\caption{On the leftmost diagram, we have the shadow $P$ with $Sing(P)$ the dashed lines,  $\partial_{int}(P)$ colored blue, and $\partial_{ext}(P)$ colored black. We cap $\partial_{int}(P)$ with  disks to obtain regions $R_j$ shaded blue, and we mark their centers with $p_1$ and $p_2$ to recover a new shadow $P'$. In this case, the gleams of the new regions are both $\frac{1}{2}$. $P'$ is the shadow  of a link complement in $S^3$ such that the preimages of the points $p_1$ and $p_2$ are  the cores of the  $\infty$-framed link in $S^3$. Drilling out the $\infty$-framed link gives a link  complement in $S^3$  homeomorphic  to the  complement of the original fundamental shadow link.}
\label{Fig:ShadowEx}
\end{figure}

\begin{proof}
Let $P$ be a boundary-decorated shadowed polyhedron of a $4$-dimensional manifold $W$ with  $\partial_{int} (P)= L$  a link. Performing Dehn filling along $L$ in the manifold $\partial W$ corresponds to gluing  $2$-handles $B^2 \times B^2$ to the manifold $W$. Without loss of generality, we can assume the copies of $B^2 \times B^2$ collapse onto   disks $B^2 \times \{pt\}$ such that the boundary of each disk $S^1 \times \{pt\}$ is a component of $\partial_{int} (P)$. This gives a new simple polyhedron $P'$  which is the shadow of a manifold $M'$ where the gleam is determined by the framing of the link.
\end{proof}

From  Lemma \ref{Cap}, we can construct families of links in $S^3$ which satisfy   Conjecture \ref{VolC}. We will proceed with the proof of Theorem \ref{Families}.

\begin{proof}[Proof of Theorem \ref{Families}]
The families of links given in Figure \ref{Fig:Families} can be obtained as follows. For reference, we demonstrate the case for $K_1$ in Figure \ref{Fig:ShadowEx}.

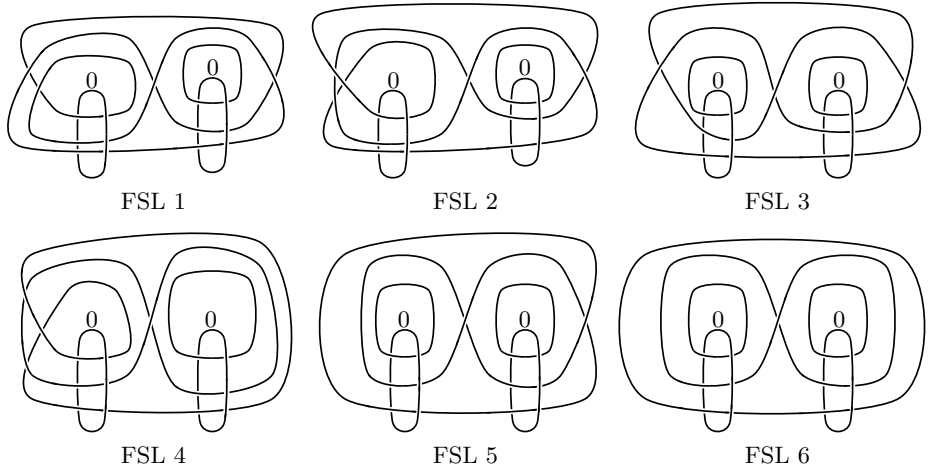
\begin{figure}
\labellist
\pinlabel \small{$0$} at 49 182
\pinlabel \small{$0$} at 106 188
\pinlabel \small{$0$} at 191 182
\pinlabel \small{$0$} at 253 188
\pinlabel \small{$0$} at 344 182
\pinlabel \small{$0$} at 401 182
\pinlabel \small{$0$} at 49 69
\pinlabel \small{$0$} at 105 69
\pinlabel \small{$0$} at 196 69
\pinlabel \small{$0$} at 253 69
\pinlabel \small{$0$} at 344 69
\pinlabel \small{$0$} at 401 69
\pinlabel \small{FSL 1} at 78 125
\pinlabel \small{FSL 2} at 225 125
\pinlabel \small{FSL 3} at 372 125
\pinlabel \small{FSL 4} at 78 5
\pinlabel \small{FSL 5} at 225 5
\pinlabel \small{FSL 6} at 372 5
\endlabellist
\centering
\includegraphics[scale=.8]{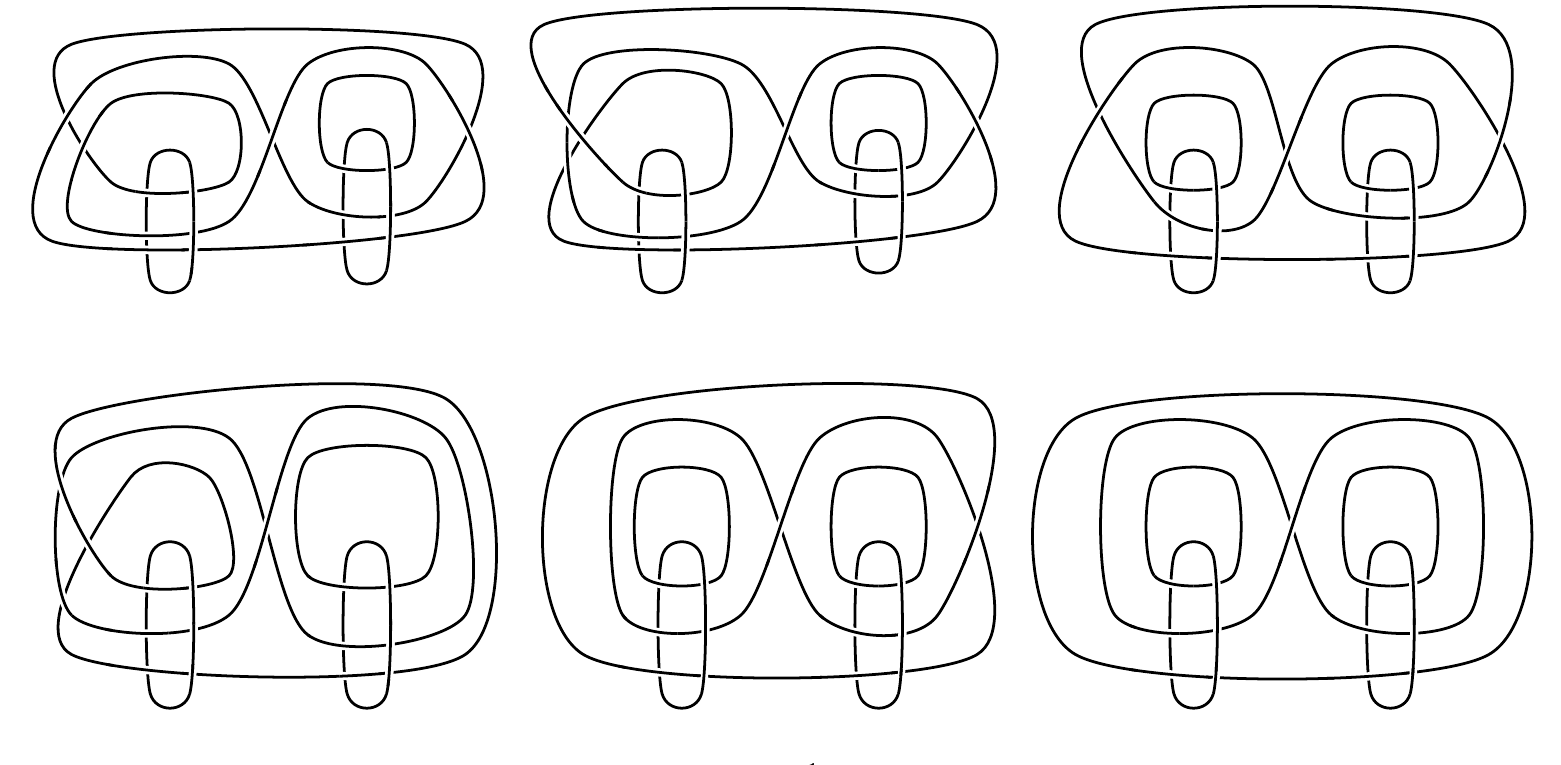}
\caption{The six fundamental shadow links with complements homeomorphic to the links up to $11$-crossings. Note that FSL 3, FSL 5 and FSL 6 correspond to the links $L_1$, $J_1$, and $K_1$ in Figure \ref{Fig:Families}, respectively.}
\label{Fig:SixEx}
\end{figure}

\begin{itemize}
\item From the construction of shadows from graphs in Section \ref{2.2}, let $S(P)$ be the simple polyhedron of the fundamental shadow link of complexity $k$ arising from the planar $4$-valent graph on the left of Figure \ref{Fig:Planar} which corresponds to a link in connected sums of $S^2 \times S^1$ on the left side of Figure \ref{Fig:Families}.   Let $P$ be the boundary-decorated shadowed polyhedron obtained from $S(P)$ such that the innermost $(k+1)$-boundary components of $P$  are colored as internal boundary components with the remaining boundaries colored as external boundary components. In the leftmost diagram of Figure \ref{Fig:ShadowEx}, this corresponds to the simple polyhedron with singular set the dashed lines and with boundary the blue and black curves. Here the internal boundary components are colored blue, and the external boundary components are colored black. 

\item Consider the new shadow $P'$ obtained from capping the internal colored boundary components by  disks. Label these new regions as $R_j$. From Lemma \ref{Cap}, this corresponds to performing a Dehn filling along the internal boundary components where the gleam of the resultant region is determined by the framing. We can choose a framing such that the gleams of $R_j$ are $\frac{V_{R_j}}{2}$ where $V_{R_j}$ is the number of vertices touching $R_j$. In the second diagram of Figure \ref{Fig:ShadowEx}, this can be seen by capping the internal colored boundary components with the blue shaded disks with gleams $\frac{1}{2}$. As in Section \ref{2.4},  $P'$ will be the shadowed polyhedron  for the link complement $S^3 \backslash Nbd(L)$ where the link complement comes from coloring the boundary  external as opposed to internal. From Proposition \ref{Stein}, we can choose a  map $f$ with a Stein surface $W_f$ and Stein factorization $f=g \circ h$ such that $W_f$ can be identified with $P'$ where the preimages of points are as described in Proposition \ref{Stein}.

\item Performing the Dehn filling corresponds to gluing copies of the $2$-handles $B^2 \times B^2$ to the $4$-thickening of the initial boundary-decorated shadowed polyhedron $P$. Let $p_j \in R_j$ be the point in the $2$-handle that corresponds to the center of the capped disk. The preimage of the centers of these disks under the  map $h$ corresponds to  cores of  tori. From Proposition \ref{Stein}, these preimages run parallel to the $z$-axis to form an unknotted link $\{p_j\} \times S^1$ embedded in $S^3 \backslash Nbd(L)$ with framing $\infty$ and with core   $L_0$.  This is illustrated in the third diagram of Figure \ref{Fig:ShadowEx}.

\item We now remove the $2$-handles which contain the framed link  leaving additional boundary tori.  The resultant manifold is ${S^3 \backslash Nbd(L \sqcup L_0)}$  with boundary-decorated shadowed polyhedron $P''$ where $P''$ has $S(P)$ as the simple polyhedron and with boundary components  colored as external. 

\item  As stated in Section \ref{2.2}, $P''$ is a boundary-decorated shadowed polyhedron for the complement of a fundamental shadow link, therefore $L \sqcup L_0 \subset S^3$ is the realization of the fundamental shadow link in $S^3$. This is the link  shown on the rightmost diagram of Figure \ref{Fig:ShadowEx}.  Since the complements of the fundamental shadow links of complexity $k$ satisfy Conjecture \ref{VolC},  the complement of $L \sqcup L_0$ in $S^3$ also satisfies the conjecture and has volume $2kv_8$ by Proposition \ref{FSLVol}.
\end{itemize}

\end{proof}

\begin{remark}\label{Rem:DehnFill}
Notice by performing a Dehn filling of slope $\infty$  on the rightmost of the $(k+1)$ unknotted components bounding twice punctured disks, we can reduce our link from a fundamental shadow link of complexity $k$ to one of complexity $(k-1)$ for the links $J_k$ and $K_k$. 
\end{remark}

\begin{remark} \label{Rem:PropOfLink}
For the links appearing in Theorem \ref{Families}, we used the property that each $0$-framed unknot could be paired with a component of the fundamental shadow link to form a  Hopf link.  In the  manifold of connected sums of $S^2 \times S^1$, the fundamental shadow link component of the Hopf link pair is isotopic to the core of the $0$-framed unknot component after the Dehn surgery. This implies that the complement of the fundamental shadow link can be viewed in $S^3$, up to homeomorphism,  by removing the corresponding fundamental shadow link components from the  Hopf links  and by drilling out all the $0$-framed unknot components. This can be easily seen by comparing the left diagrams with the right diagrams of Figure \ref{Fig:Families}.   This process works for any fundamental shadow link where we can pair up each of the $0$-framed unknot components with  fundamental shadow link components so that the resulting Hopf links are disjoint from each other. 

Although the  Hopf links  can be viewed immediately in the families of Theorem \ref{Families} from the maximal tree construction of the fundamental shadow link given in Section \ref{2.2}, we will see that this is generally not the case. As we will see in Theorems \ref{Fibered} and \ref{Sublink}, the Hopf links may not be disjoint from each other and require further anaylsis using Kirby calculus. Because of  this, we chose to utilize the Stein map approach in proving Theorem \ref{Families}. The Stein surface is equipped with a particular  collapsing map which allows for a more natural approach to obtain the  Hopf links when we generalize to the cases in Theorems \ref{Fibered} and \ref{Sublink}. By using the  Stein map, we find that each capped disk corresponds to a Hopf link where the boundary of the disk is the fundamental shadow link component and the preimage of the center of the disk is the $0$-framed unknot component. For this method to work, we need a shadow of a fundamental shadow link that can be augmented by capping with disks to a  shadow of the form $P_L$ as defined in Section \ref{2.4}.  In these cases, the $0$-framed unknots run parallel to the $z$-axis implying all the Hopf links are  disjoint from each other.  
\end{remark}

\subsection{Links up to 11-crossings}\label{3.1}

\begin{figure}
\centering
\includegraphics[scale=.9]{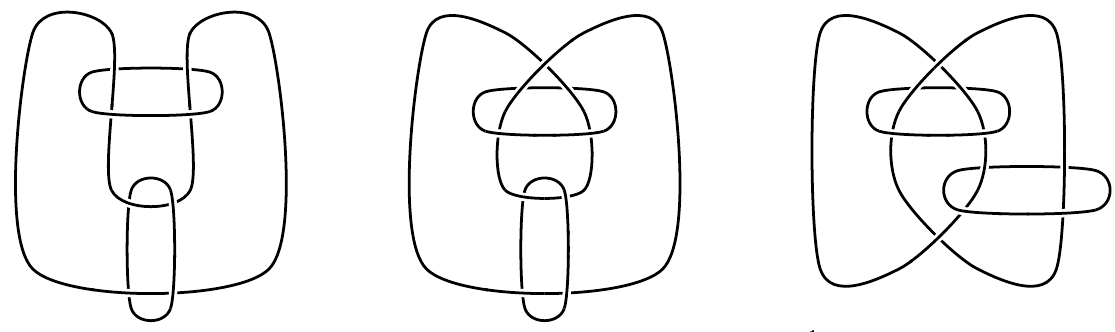}
\caption{The Borromean rings and the Borromean twisted sisters.}
\label{Fig:TwistedSis}
\end{figure}

We will now categorize from \textit{LinkInfo} which octahedral links of volume $2v_8$ up to 11-crossings are complexity one fundamental shadow links. Figure \ref{Fig:SixEx} has the six fundamental shadow links with   homeomorphic complements to the links.  As mentioned in the introduction, the only link \textit{SnapPy} determined was not homeomorphic to a fundamental shadow link was  $L_{10n59}$. This could be due to a numerical error and is not rigorous, and the link  may still be a fundamental shadow link of complexity one.

\begin{table}[h]
\centering
\begin{tabular}{|c|c|}
\hline
 Fundamental Shadow Link  &  LinkInfo Name \\
 \hline
 \hline
 FSL 1 & $L_{10n32}$ \\
 \hline
 FSL 2 & $L_{10n36}$ \\
 \hline
 FSL 3 & $L_{6a4}$, $L_{9n25}$, $L_{11n287}$, $L_{11n378}$ \\
 \hline
 FSL 4 & $L_{10n84}$, $L_{10n87}$ \\
 \hline
 FSL 5 & $L_{8n5}$, $L_{9n26}$, $L_{10n70}$, $L_{11n376}$, $L_{11n385}$ \\
 \hline
 FSL 6 & $L_{8n7}$, $L_{10n97}$, $L_{10n105}$, $L_{10n108}$ \\
 \hline
\end{tabular}
\caption{Links with homeomorphic complements.}
\label{Table:Links}
\end{table}

\begin{thm}\label{TableLinks}
The links  given in Table \ref{Table:Links}  satisfy  Conjecture \ref{VolC} and have volume $2v_8$. 
\end{thm}

\begin{proof}
We will prove  Theorem \ref{TableLinks} by using the following steps:

\begin{itemize}
\item Since the links are complexity one fundamental shadow links, then the volume of the links are $2v_8 \approx 7.3277$. Using \textit{LinkInfo}, we can find all links up to $11$-crossings whose volume are approximately $2v_8$.
\item This gives us the list of links: $L_{6a4}$, $L_{8n5}$, $L_{8n7}$,  $L_{9n25}$,  $L_{L9n26}$, $L_{10n32}$, $L_{10n36}$, $L_{10n59}$, $L_{10n70}$, $L_{10n84}$, $L_{10n87}$, $L_{10n97}$,  $L_{10n105}$, $L_{10n108}$, $L_{11n287}$,  $L_{11n376}$, $L_{11n378}$, and $ L_{11n385}$ where the first number indicates the number of crossings, the letter indicates whether it is alternating or non-alternating, and the second number indicates it is the $n$th link with the same first two parameters.
\item The fundamental shadow links obtained by replacing edges by $3$-braids of single vertex planar $4$-valent graphs are homeomorphic if they induce the same permutation on the $3$-strands, therefore, there are only finitely many fundamental shadow links with volume $2v_8$. 
\item As in Section \ref{2.2}, we can realize the complement of these fundamental shadow links as the complement of a link in $S^3$ where we perform a $0$-framing Dehn surgery on a particular number of its components. As stated earlier, the fundamental shadow links with complements homeomorphic to complements of links up to $11$-crossings are shown in Figure \ref{Fig:SixEx}. We can define each of these manifolds in \textit{SnapPy}, as well as the complements of the possible octahedral links from \textit{LinkInfo}.
\item Finally, we can use \textit{SnapPy} to test the finitely many cases to determine which of these manifolds are homeomorphic.
\end{itemize}

\end{proof}

Now by using Theorem \ref{Families} for $k=1$ and Theorem \ref{TableLinks}, we can give an alternative proof to Conjecture \ref{VolC} with the Borromean rings in addition to the Borromean twisted sisters in Figure \ref{Fig:TwistedSis}. In this sense, the link families of Theorem \ref{Families} are a generalization of the three links.

\begin{cor}\label{Borromean}
The Borromean rings and the Borromean twisted sisters shown in Figure \ref{Fig:TwistedSis} are hyperbolic with volume $2v_8$ and satisfy Conjecture \ref{VolC}. 
\end{cor}

\begin{proof}
Consider the case $k=1$ for Theorem \ref{Families}. Conjecture \ref{VolC} is known for FSL 3, FSL 5, and FSL 6 as they correspond to the links $L_1$, $J_1$, and $K_1$, respectively.   From Table \ref{Table:Links}, the links $L_1$, $J_1$, and $K_1$  have complements homeomorphic to the Borromean rings and the Borromean twisted sisters shown in Figure \ref{Fig:TwistedSis}, respectively. 
\end{proof}

\begin{remark}\label{Rem:ComplexityTwo}
All of the links in \textit{LinkInfo} have volume strictly less than $6v_8$, therefore up to $11$-crossings, there are only link complements homeomorphic to complements of fundamental shadow links of complexity at most two. As discussed in Theorem \ref{TableLinks}, we found possible links of complexity one; however, we could also check possible links of complexity two in \emph{LinkInfo}. The  candidates of these links are the links $L_{11n387}$ and $L_{11n388}$; however, it is not determined if either of these have complements homeomorphic to the complement of a complexity two fundamental shadow link. 
\end{remark}


\subsection{Fibering over $S^1$}\label{3.2}

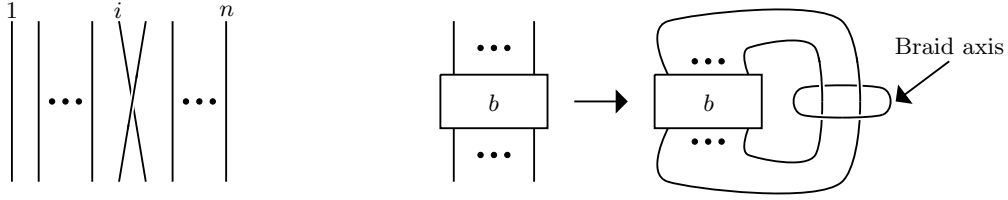
\begin{figure}
\labellist
\pinlabel \small{$1$} at 0 82
\pinlabel \small{$i$} at 45 82
\pinlabel \small{$n$} at 91 82
\pinlabel \small{$b$} at 204 43
\pinlabel \small{$b$} at 295 43
\pinlabel \small{Braid axis} at 397 67
\endlabellist
\centering
\includegraphics[scale=.9]{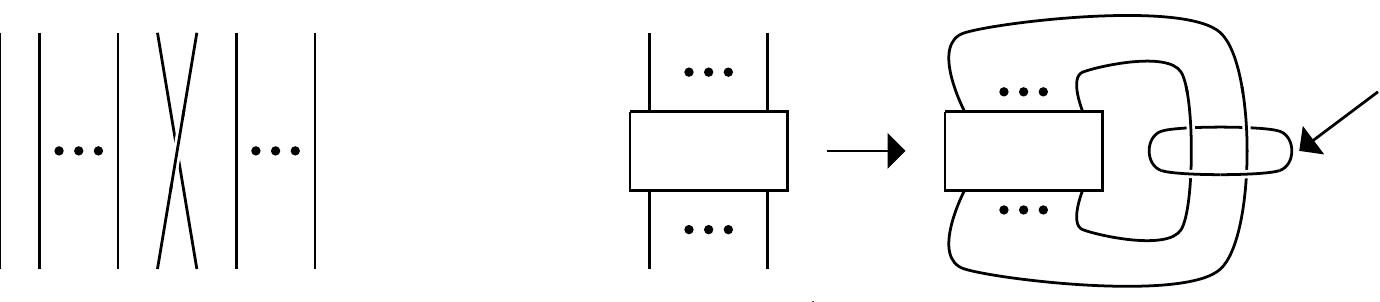}
\caption{ On the left, we have the generator $\sigma_i$ for the braid group $B_n$ where $i \in \{1, \dots, n-1\}$. On the far-right for a given element $b \in B_n$, we have the braided link $\overline{b}$ which is the braid closure with its braid axis. }
\label{Fig:Braid}
\end{figure}

 All the examples of fundamental shadow links so far have been constructed using the planar graphs in the left of Figure \ref{Fig:Planar}; however, we could have also used other planar graphs such as in the right of Figure \ref{Fig:Planar} to construct  examples of links. Using these particular planar graphs, we can construct a family of links  in $S^3$ that fiber over $S^1$ with fiber disks with punctures and with explicit monodromies. 
 
 As recalled in the introduction, 
 let $B_n$ be the braid group of $n$-strands with standard generators $\sigma_i$ for $i \in \{1,\dots, n-1\}$ where $\sigma_i$ is illustrated as in the left diagram of Figure \ref{Fig:Braid}. We have the following definition.
 
 \begin{definition}\label{Def:BraidedLink}
 For a given braid $b \in B_n$, a \emph{braided link}, denoted by $\overline{b}$, is the closure of $b$ in $S^3$ with an additional unknotted component known as the \emph{braid axis} shown on the far-right diagram of Figure \ref{Fig:Braid}.
 \end{definition}

  We will see that for a given $\overline{b}$,
 our link  in $S^3$ fibers over $S^1$ with fiber surface the punctured disk bounded by the braid axis.
  
Recall that the mapping class group of an orientable surface  $MCG(\Sigma)$  is the group of isotopy classes of orientation preserving self-homeomorphisms of $\Sigma$ where the group operation is composition of self-homeomorphisms. Now there exists a surjective homomorphism from the braid group of $n$-strands to the mapping class group of the $n$-punctured disk $MCG(D_n)$
$$\Gamma: B_n \to MCG(D_n)$$
 which sends the  generator $\sigma_i\in B_n$, for $i \in \{1,\dots,n-1\}$, to an element of the mapping class group. This element is represented by a positive half-twist on an arc between the $i$-th puncture to the $(i+1)$-th puncture. The kernel for such a map is generated by the element $(\sigma_1 \sigma_2 \dots \sigma_{n-1})^n$ which corresponds to a full-twist along the boundary. More details on mapping class groups can be found in \cite{FarM12} by Farb and Margalit. 
 
 For a given $\Gamma(b)\in MCG(D_n)$, we can construct the mapping tori
 $M(\Gamma(b)) =  D_n \times [0,1] / (x,0) \sim (\Gamma(b)(x),1)$ with monodromy $\Gamma(b)$. The mapping tori have homeomorphic complements with the braided link $\overline{b}$. 
 
 Now we consider the family of braids $\{b_k\}$ defined in the introduction. We reiterate in the following: 
 
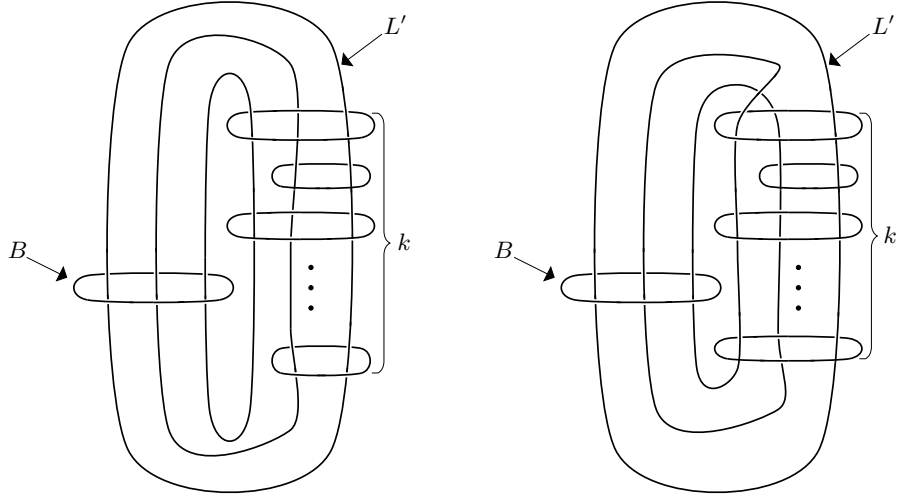
\begin{figure}
\labellist
\pinlabel \small{$B$} at -4 121
\pinlabel \small{$B$} at 225 121
\pinlabel \small{$k$} at 178 125
\pinlabel \small{$k$} at 407 128
\pinlabel \small{$L'$} at 174 227
\pinlabel \small{$L'$} at 404 227
\endlabellist
\centering
\includegraphics[scale=.8]{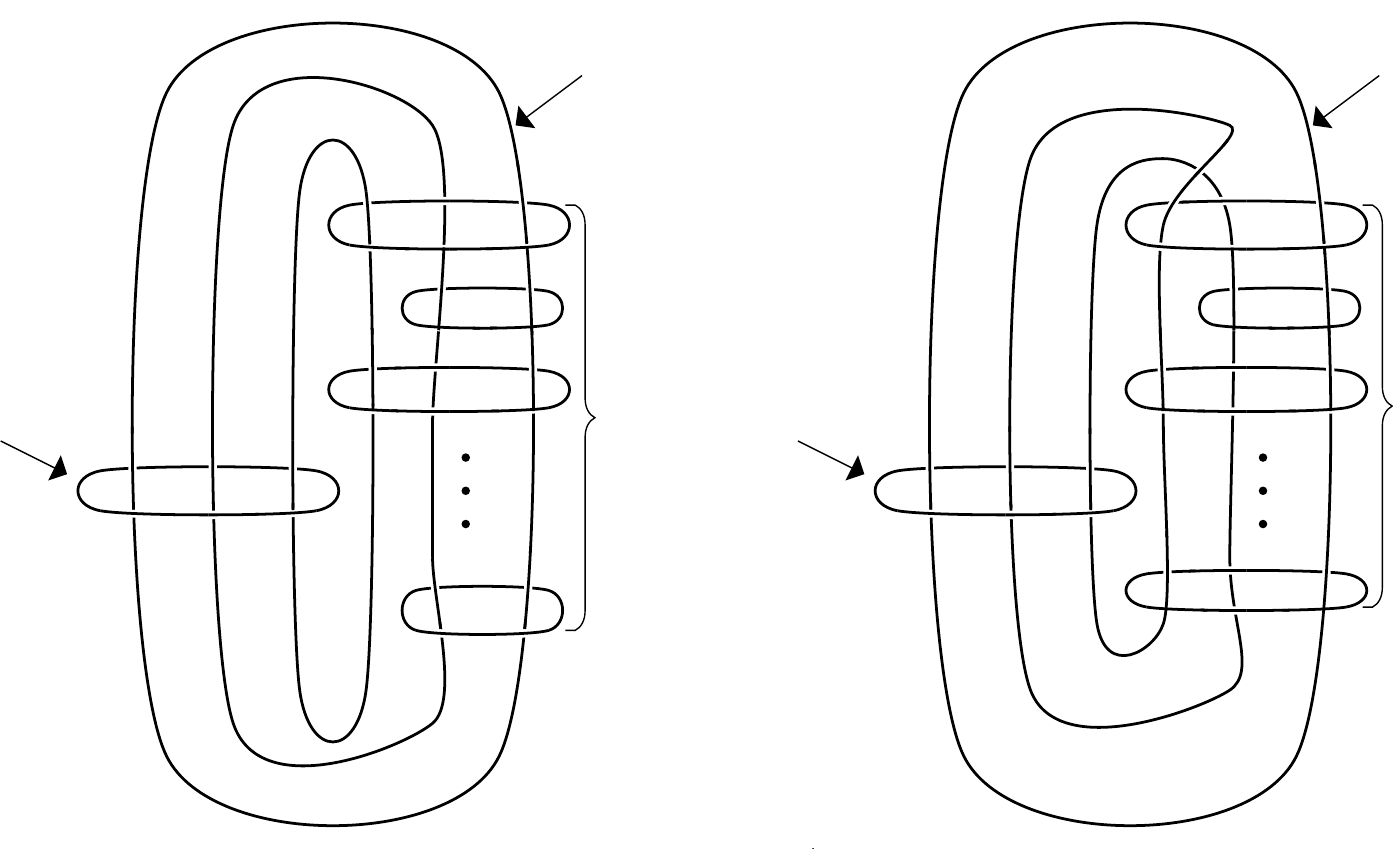}
\caption{Links in $S^3$ for complexity $k$ fundamental shadow links arising from replacing the edges of the rightmost planar graph of Figure \ref{Fig:Planar} with identity $3$-braids.  On the left for when $k$ is even, there are $(k+4)$ components and on the right when $k$ is odd, there are $(k+3)$ components. The braid axis component is labeled as $B$, and the component that bounds a $(k+1)$-punctured disk is labeled $L'$. }
\label{Fig:Fibered}
\end{figure}
 
 \begin{itemize}
\item For $k=1$, define $b_k:=\sigma_1^{-2} \sigma_2^{2} \in B_3$.
\item For $k=2m$, define $b_k\in  B_{k+3}$ by
 $$ b_k:=\prod_{i=1}^{m} (\sigma_{2i+1}\sigma_{2i} \cdots \sigma_3 \sigma_2 \sigma_1^2  \sigma_2^{-1} \sigma_3 \sigma_4 \cdots \sigma_{2i}
 \sigma_{2i+1}) (\sigma_{2i+2} \sigma_{2i+1} \cdots \sigma_3 \sigma_2^2 \sigma_3  \cdots \sigma_{2i+1} \sigma_{2i+2} ).$$
\item For $k=2m-1$, define $b_k\in  B_{k+3}$  by
$$b_k:=b\ \prod_{i=2}^{m} (\sigma_{2i} \sigma_{2i-1} \cdots \sigma_3 \sigma_2^2 \sigma_3 \cdots \sigma_{2i-1} \sigma_{2i}) (\sigma_{2i+1} \sigma_{2i} \cdots \sigma_3 \sigma_2 \sigma_1^2 \sigma_2^{-1} \sigma_3 \sigma_4 \cdots \sigma_{2i} \sigma_{2i+1}), $$
where $b=(\sigma_1)(\sigma_3 \sigma_2 \sigma_1^2 \sigma_2^{-1} \sigma_3).$ 
\end{itemize}
 
 From the relation between the mapping tori of  $MCG(D_n)$ and the braided link, we can construct links in $S^3$ that fiber over $S^1$ with the explicit monodromies $\Gamma(b_k)$. 
 In particular, the complements of the braided links $\overline{b_k}$ have the properties stated in the following.

 \begin{thm}\label{Fibered}
For $k\geq 1$,  let $M_k:=S^3\setminus \overline{b_k}$ denote the complement of the braided link $\overline{b_k}.$ Then
\begin{enumerate}  
\item $M_k$ fibers over  $S^1$  with fibered surface a disk with punctures,
\item $M_k$ is hyperbolic with volume $2kv_8,$ and
\item $M_k$ satisfies  Conjecture \ref{VolC}.
\end{enumerate}
\end{thm}

\begin{figure}
\labellist
\pinlabel \small{$B$} at 312 56
\pinlabel \small{$B$} at 392 26
\endlabellist
\centering
\includegraphics[scale=.8]{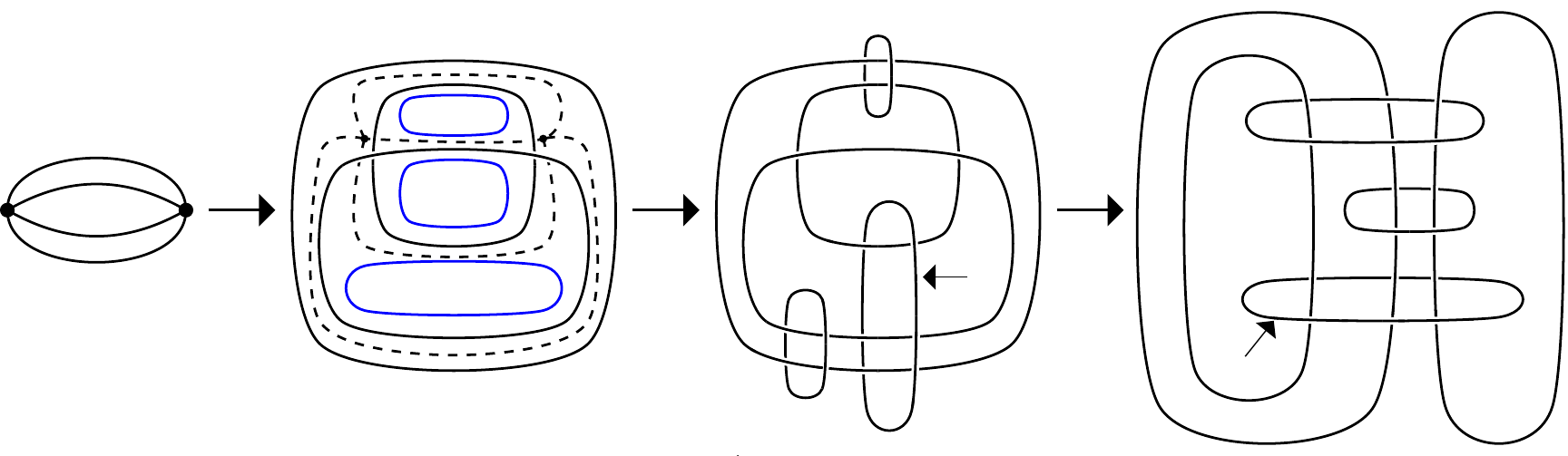}
\caption{On the leftmost diagram, we have a planar $4$-valent graph with complexity $2$. From the graph, we obtain a boundary-decorated shadowed polyhedron where the internal colored boundaries are blue, the external black, and the singular set dashed. As in Theorem \ref{Families}, we cap the blue curves with disks and drill out tori to recover the link in the third diagram where the middle blue curve corresponds to the link component $B$. Through an isotopy, we obtain the link in the fourth diagram.    }
\label{Fig:Cycle}
\end{figure}

\begin{proof}
As noted in Remark \ref{Rem:PropOfLink}, the maximal tree construction of the fundamental shadow link given in Section \ref{2.2} will not result in disjoint Hopf links which are formed from  a $0$-framed unknot   with a fundamental shadow link component. In the case of Theorem \ref{Fibered}, we will see that precisely two of the Hopf links  are not disjoint from each other. By utilizing the Stein map approach,  each capped disk will be used to construct the disjoint Hopf links avoiding the problem arising from the maximal tree approach. Alternatively, we will also demonstrate how to obtain the same disjoint Hopf links  using the maximal tree construction with an application of Kirby calculus.

Consider the case for $k \geq 2$ for the the fundamental shadow link of the planar graph on the right of Figure \ref{Fig:Planar}. As before, we replace each vertex with six strands as described in Section \ref{2.2}, and we choose to replace the edges with the identity $3$-braids to obtain a shadow $P$. In particular, we will replace the vertices by Figure \ref{Fig:Graph} or its mirror in a specific pattern which will allow us to simplify the resulting link. This choice will be made explicit later in the proof. 

Now we color the outermost boundary of the shadow as external, the boundary components with crossing information external, and the remaining boundaries as internal.  From the same construction as in Theorem \ref{Families}, we can cap components of the shadow by disks to obtain a shadow of the form $P_L$ as defined in Section \ref{2.4}. As mentioned in Remark \ref{Rem:PropOfLink}, we will obtain a collection of disjoint Hopf links where one component is a fundamental shadow link component and the other a $0$-framed unknot component. From this construction, we will obtain either the left link for $k$ even or the right link for $k$ odd in Figure \ref{Fig:Fibered}. We will illustrate an example for $k=2$ in Figure \ref{Fig:Cycle}. Because the vertices of the planar graph are arranged cyclically, there exists an inner region  of the modified shadow $P'$ where the corresponding link component  bounds a $3$-punctured disk of the fundamental shadow link in $S^3$. This is the component we label as $B$ in Figure \ref{Fig:Fibered}, and it will eventually be the braid axis for our link.  

\textit{Alternate argument for constructing the link}: Alternatively, we can obtain the link from using the maximal tree construction of Section \ref{2.2} and Kirby calculus. We will again begin with the planar graph on the right of Figure \ref{Fig:Planar}. Note that a  graph of $k$ vertices will have a maximal tree of $(k-1)$ edges. Now we choose our maximal tree $T$ to be  $(k-1)$ of the $k$ edges which bound the innermost region where we denote the excluded edge by $E$. We now follow the maximal tree construction where we replace each edge and vertex as stated previously  in the Stein map approach. We can see that the $0$-framed unknot resulting from the edge $E$ will not be disjoint from one of the other Hopf links. By performing a handleslide  of the $0$-framed unknot corresponding to edge $E$ over the $0$-framed unknot from the other Hopf link, we will arrive to the same surgery presentation that was obtained from the Stein map approach. In particular, we can see that each of the $0$-framed unknots can be isotoped to be parallel to the $z$-axis, thus disjoint from each other. Since each of the Hopf links  are disjoint, we can obtain the link in $S^3$ from the same process mentioned in Remark \ref{Rem:PropOfLink}. This concludes the alternative argument using the maximal tree  approach in constructing the link in $S^3$.

 We will now continue with the proof of the theorem. For $k$ odd, the sublink consisting of the components with crossings given by the vertices of the graph is a knot. In the case when $k$ is even,  the sublink is a two component link. This is due to the sublink being realized as the closure of a $2$-braid where each crossing contributes a generator of the group. We choose to replace the vertices of the graph by Figure \ref{Fig:Graph} or its mirror such that the resulting crossings of the closure of the  $2$-braid sublink alternate between a generator of the group and its inverse. When $k$ is even, we can reduce the $2$-braid to the identity as in Figure \ref{Fig:Cycle}, and when $k$ is odd, we can reduce the $2$-braid to a single generator. This  accounts for the difference between the number of components between the cases odd and even in Figure \ref{Fig:Fibered}.

Now let $L$ be a link of Figure \ref{Fig:Fibered} with link component $L'$ as labeled. $L'$ bounds a $(k+1)$-punctured disk where one of the punctures comes from the component $B$. Now consider the self-homeomorphism obtained by performing a full-twist along the $(k+1)$-punctured disk that $L'$ bounds. Under this self-homeomorphism, we obtain a new link in $S^3$ with complement a manifold $M_k$ where $B$ now bounds a $(k+3)$-punctured disk $D_{k+3}$ as the braid axis. The components of our new link in $S^3$ can now be isotoped to the braided link  $\overline{b_k}$.  Since we began with a fundamental shadow link, our manifold $M_k$ will be hyperbolic with volume $2kv_8$ and satisfy Conjecture \ref{VolC}.

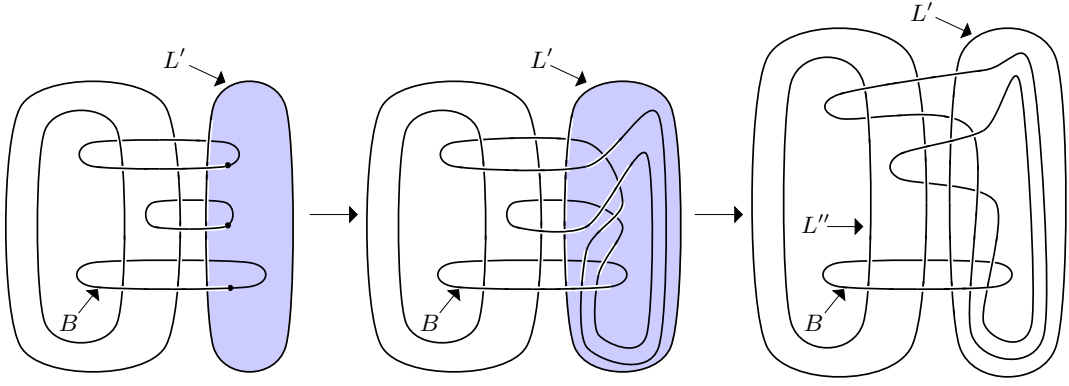
\begin{figure}
\labellist
\pinlabel \small{$L'$} at 82 155
\pinlabel \small{$L'$} at 255 155
\pinlabel \small{$L'$} at 434 177
\pinlabel \small{$L''$} at 384 77
\pinlabel \small{$B$} at 32 30
\pinlabel \small{$B$} at 202 30
\pinlabel \small{$B$} at 383 30
\endlabellist
\centering
\includegraphics[scale=.8]{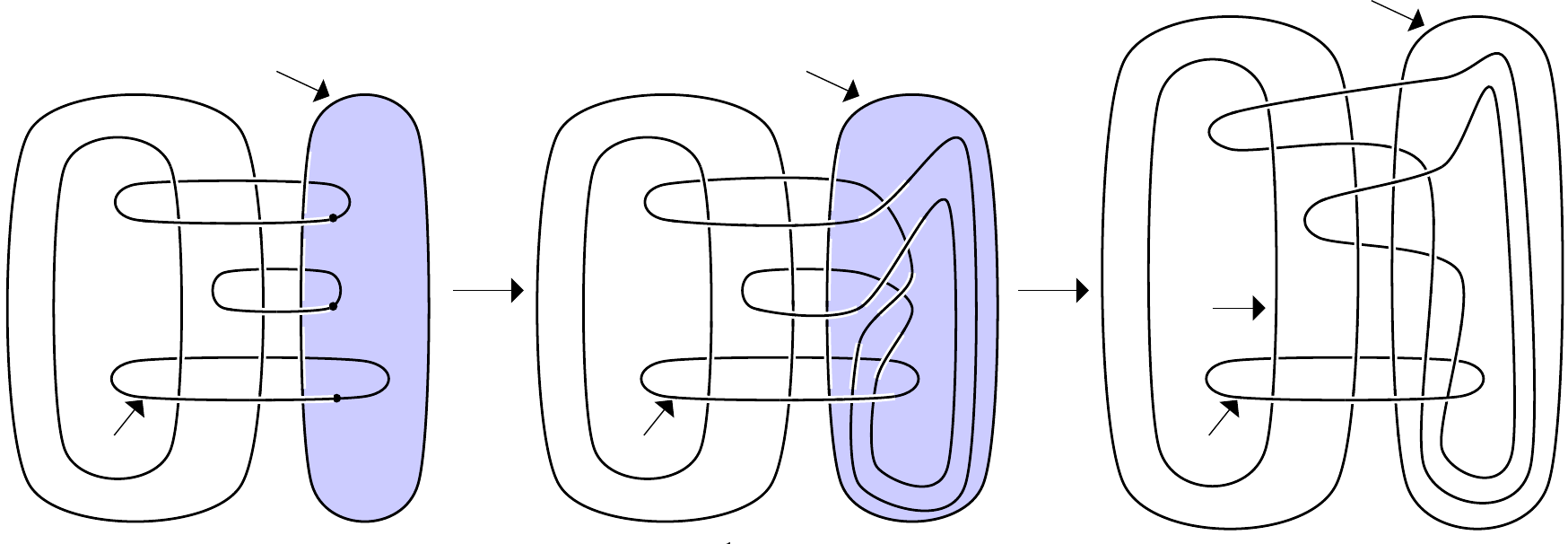}
\caption{The case of obtaining a fibration for $M_2$. Performing a  full-twist on the $3$-punctured disk shaded blue in the leftmost diagram results in the middle diagram. Through an isotopy of the link, we obtain a link fibering over $S^1$ with fiber $D_5$ and monodromy $(\sigma_3\sigma_{2} \sigma_1^{2} \sigma_2^{-1} \sigma_3)(\sigma_4 \sigma_3 \sigma_2^{2} \sigma_3 \sigma_4)$ represented by the last diagram.    }
\label{Fig:FiberedEx}
\end{figure}

If we take the surjective homomorphism $\Gamma$ between $B_{k+3}$ and $Mod(D_{k+3})$, the manifold $M_k$ fibers over $S^1$  as the mapping tori
$M_k = M(\Gamma(b_k))= (D_{k+3} \times [0,1]) / ((x,0)  \sim (\Gamma(b_k)(x), 1)  )$
where the monodromy  is given as the image under $\Gamma$ of the following element in $B_{k+3}$:

$$ \prod_{i=1}^{m} (\sigma_{2i+1}\sigma_{2i} \cdots \sigma_3 \sigma_2 \sigma_1^2  \sigma_2^{-1} \sigma_3 \sigma_4 \cdots \sigma_{2i} \sigma_{2i+1}) (\sigma_{2i+2} \sigma_{2i+1} \cdots \sigma_3 \sigma_2^2 \sigma_3 \cdots \sigma_{2i+1} \sigma_{2i+2} ), $$
when $k=2m$ and 
$$(\sigma_1)(\sigma_3 \sigma_2 \sigma_1^2 \sigma_2^{-1} \sigma_3) \prod_{i=2}^{m} (\sigma_{2i} \sigma_{2i-1} \cdots \sigma_3 \sigma_2^2 \sigma_3 \cdots \sigma_{2i-1} \sigma_{2i}) (\sigma_{2i+1} \sigma_{2i} \cdots \sigma_3 \sigma_2 \sigma_1^2 \sigma_2^{-1} \sigma_3 \sigma_4 \cdots \sigma_{2i} \sigma_{2i+1}), $$
when $k=2m-1$ and $k\neq 1$. Obtaining this fibration of the link in $S^3$ with the given monodromy is straightforward.  We can isotope the braid axis $B$ so that locally we have a copy of $D_{k+3} \times [-\epsilon,\epsilon]$ for some $\epsilon$. Cutting the link along the punctured disk $B$ bounds will show that the braid element $b_k$ can be represented  as above. We show an example of this process for $k=2$ in Figure \ref{Fig:FiberedEx}.

Now let $k=1$ and consider the rightmost Borromean twisted sister in Figure \ref{Fig:TwistedSis}. As shown in Corollary \ref{Borromean}, this is the complement of a fundamental shadow link of complexity one. We can see in Figure \ref{Fig:Borromean} there exists an isotopy of the link such that one of the components is a braid axis that bounds $D_3$ where cutting along the disk results in the braid group element $\sigma_{1}^{-2} \sigma_{2}^{2} \in B_3$. 

Therefore for any $k$, we can find a manifold with volume $2kv_8$ which satisfies the Conjecture \ref{VolC},  fibers over $S^1$ with surface a disk with punctures, and has explicit monodromy. 
\end{proof}

\begin{remark}\label{Rem:ClosedBraid} 
From the links in Theorem \ref{Fibered}, we can construct more examples of links satisfying Conjecture \ref{VolC} by performing full-twists  along the punctured disks the link components bound. Through this process, a new link with possibly different punctured disk as the fibered surface may be obtained, and its monodromy may be computed explicitly similarly to Theorem \ref{Fibered}. We may also remove the braid axis with certain self-homeomorphisms to create a closed braid. For example, we can obtain a link which is the braid closure of the element $$( \sigma_4 \sigma_3 \sigma_2 \sigma_1^2 \sigma_2 \sigma_3^{-1} \sigma_4 )(\sigma_5 \sigma_4 \sigma_3^2 \sigma_4 \sigma_5) (\sigma_1 \sigma_2 \sigma_3 \sigma_4 \sigma_5^2 \sigma_4 \sigma_3 \sigma_2 \sigma_1) \in B_6.$$
   This can be seen by performing a full-twist along the component labeled $L''$ in the rightmost diagram of  Figure \ref{Fig:FiberedEx}.
   \end{remark}

\begin{remark}\label{Rem:nBraid}
As mentioned in the proof, the  links from  Theorem \ref{Fibered} contain a closed $2$-braid as a sublink. We will see in Theorem \ref{Sublink} that we can also construct  links that contain closed $n$-braids as a sublink of a fundamental shadow link in $S^3$. This will  allow us to find other manifolds fibering over $S^1$ satisfying Conjecture \ref{VolC}. In particular for Corollary \ref{Even}, we will find a family of  links obtained from fundamental shadow links which contain a closed $3$-braid. These specific  links  can be realized as a braided link from an element of the pure braid group. In addition with the braided links in Theorem \ref{Fibered}, we can create explicit pseudo-Anosov elements in the mapping class group of a genus $0$ surface with $n$-boundary components for $n \geq 4$ and even, or $n\geq9$. More details can be found in  Section \ref{5.1}.
\end{remark}

Besides the examples of complements of fundamental shadow links from Theorems \ref{Families} and \ref{Fibered}, there exists another family of links in $S^3$ where the volume is explicit. K. H. Wong showed in \cite{Won19} that certain subfamilies of Whitehead chains  satisfy  Conjecture \ref{VolC}. The Whitehead chains $W_{a,b,c,d}$ are a family of links that can be constructed as the closure of the composition of $a$-twists, $b$-belts, $c$-clasps, and $d$-mirrored clasps. In particular when $b=1$, the complement is hyperbolic with $Vol(W_{a,1,c,d}) = (c+d) v_8$. We let $\mathcal{W}$ be this set. We will show that the families of links in Theorems \ref{Families} and \ref{Fibered} are not homeomorphic to these Whitehead chains in $\mathcal{W}$.

\begin{figure}
\centering
\includegraphics[scale=.775]{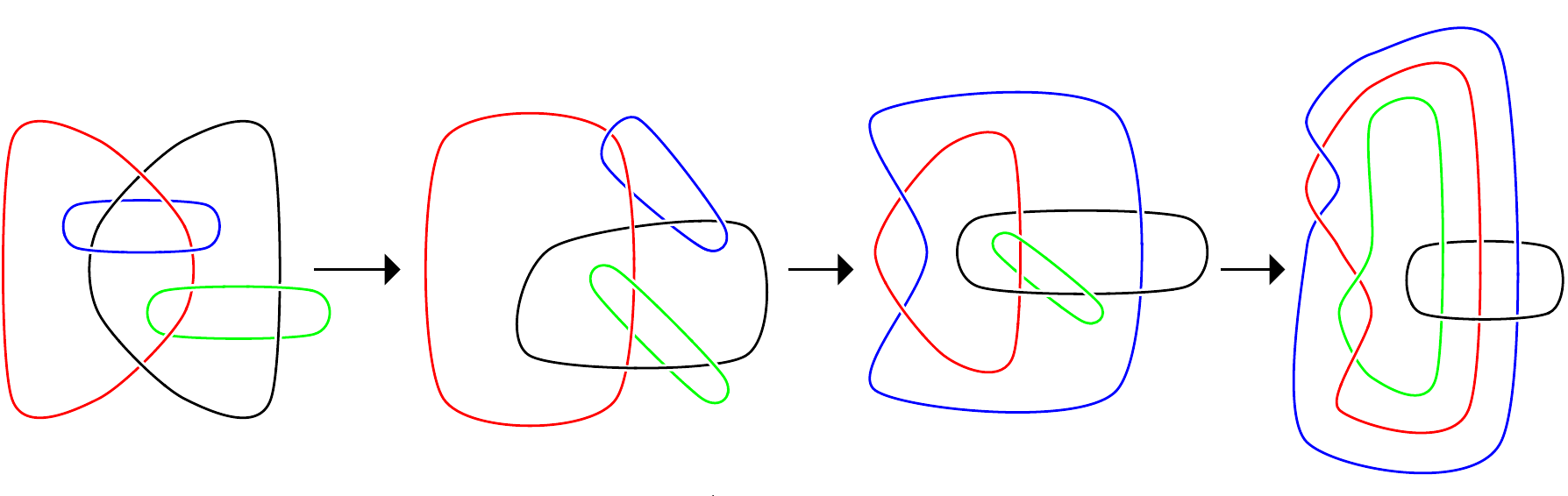}
\caption{An isotopy of the Borromean twisted sister to the closure of a braid where the braid axis is colored black. The resulting manifold has fibered surface $D_3$ over $S^1$ with monodromy the image under $\Gamma$ of  $\sigma_1^{-2} \sigma_2^{2}\in B_3$. }
\label{Fig:Borromean}
\end{figure}

\begin{prop}\label{Whitehead}
For an integer $k\geq4$, no link of  Theorems \ref{Families} and  \ref{Fibered} has complements   homeomorphic to the complement of an element of $\mathcal{W}$. 
\end{prop}

\begin{proof}
As stated previously, it is known through an explicit decomposition into ideal hyperbolic octahedra that 
$Vol(W_{a,1,c,d})= (c+d)v_8$. In the case when $(c+d)$ is even and hence the volume is an even multiple of $v_8$, the number of components of the Whitehead chain is $(c+d+1)$. This is due to each clasp and mirror clasp adding a component with the addition of the belt. If we have a fundamental shadow link of complexity $k$, then the manifold will have volume $2kv_8$. In the examples given in Theorems \ref{Families} and \ref{Fibered}, the number of components of the link will either be $(k+2)$, $(k+3)$, or $(k+4)$. For these particular Whitehead chains to have the same volume, they would require $(2k+1)$ components. This implies that for $k\geq 4$, the Whitehead chain complement will have more boundary components than the  complements of the fundamental shadow links; thus, they are not homeomorphic. 
\end{proof}

Now we consider a slightly reformulated version of Conjecture \ref{VolC} for manifolds that are not necessarily hyperbolic. The conjecture states the asymptotics of the Turaev-Viro invariant approach the Gromov norm where the conjecture is formally stated in \cite{DetKY18}. By using the explicit descriptions of the monodromies given in Theorem \ref{Fibered}, we can create manifolds that satisfy the conjecture in the following.

\begin{cor}\label{Gromov}
Let $M_k$ be a manifold from Theorem \ref{Fibered} of volume $2kv_8$ that is realized as the complement of the braided link  $\overline{b_k}$ for $b_k \in B_n$ for some $n$. Then for any $p \geq 2$, there exists a manifold $M'_k$ obtained as the complement of the braided link $\overline{b_k'}$ for $b_k' \in B_{n+p-1}$ such that 
\begin{align}\label{Eq:Gromov}
\lim_{r\to \infty} \frac{2 \pi}{r}\log| TV_q(M'_k)|  = v_3 \|M'_k\| = 2kv_8
\end{align}
where $r$ is odd, $q=\exp\left(\frac{2 \pi i}{r}\right)$, $\|M'_k\|$ denotes the Gromov norm of $M'$, and $b_k'$ is obtained by embedding $b_k$ into $B_{n+p-1}$.
\end{cor}

\begin{proof}
From Corollary $8.4$ in \cite{DetK20}, gluing a boundary component of a manifold $M$ that satisfies the first equality sign of Equation (\ref{Eq:Gromov}) to a certain boundary component  of a manifold known as the invertible cabling space results in a new manifold $M'$ which also satisfies the first equality sign of Equation (\ref{Eq:Gromov}). If the manifold $M$ satisfies Conjecture \ref{VolC}, then the asymptotics of the Turaev-Viro invariants of $M$ and $M'$ are equal. In particular, the manifold $S_p$  which is the complement of a solid torus with $p \geq 2$ parallel copies of the core is an invertible cabling space. This manifold can be viewed as complement of the braided link obtained from the identity braid in $B_p$. Since $M_k$ is the complement of a braided link, then $M_k$ can be viewed as the complement of a solid torus containing the braid closure of $b_k$. By gluing the outer torus boundary of $M_k$ to one of the inner tori of $S_p$, we obtain the manifold $M'_k$. 
\end{proof}

\section{All links are sublinks of fundamental shadow links}\label{4}

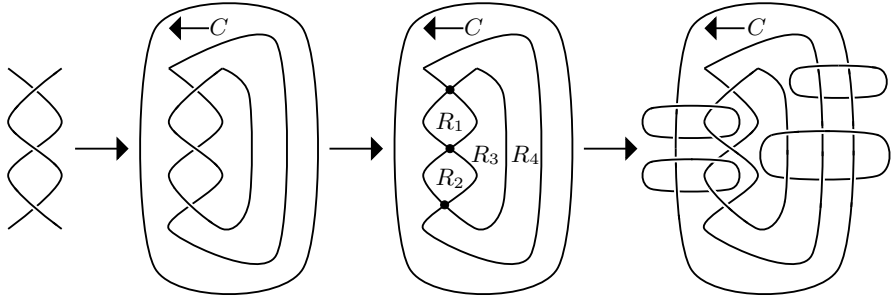
\begin{figure}
\labellist
\pinlabel \small{$C$} at 90 117
\pinlabel \small{$C$} at 198 117
\pinlabel \small{$C$} at 318 117
\pinlabel \small{$R_1$} at 188 77
\pinlabel \small{$R_2$} at 188 53
\pinlabel \small{$R_3$} at 203 63
\pinlabel \small{$R_4$} at 220 63
\endlabellist
\centering
\includegraphics[scale=.9]{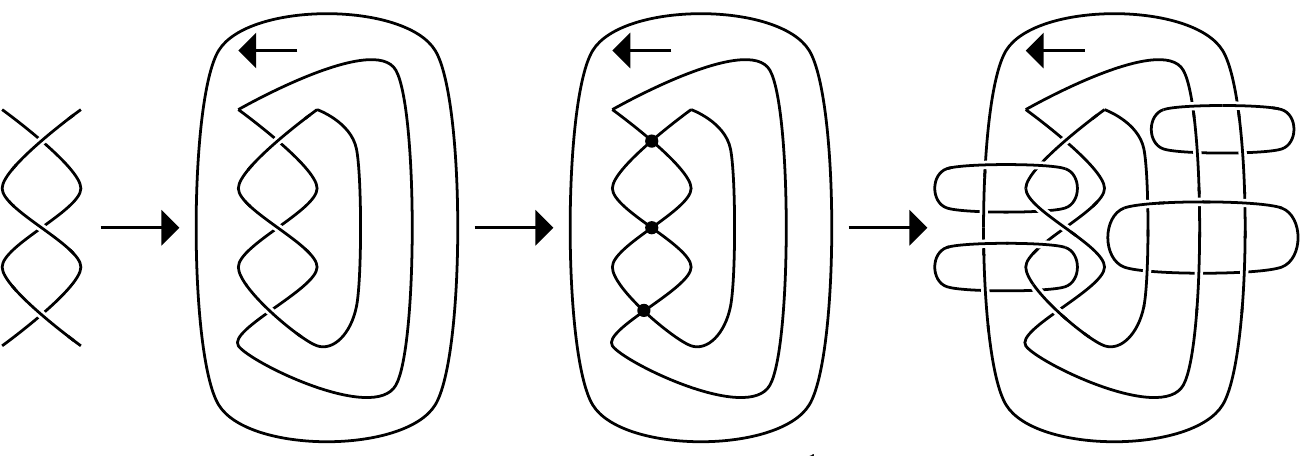}
\caption{We begin with the unreduced word $b=\sigma_1 \sigma_1^{-1} \sigma_1^{-1} \in B_2$ which we embed into $B_3$  shown in the first two diagrams with outermost component labeled $C$. The projection of the second diagram onto $S^2$ gives the third diagram where we label the regions not touching $C$ with $R_i$. Each region $R_i$ corresponds to a link component parallel to the $z$-axis given in the final diagram. The final diagram is of a link that has volume $6v_8$, satisfies Conjecture \ref{VolC}, and contains $\hat{b}$ as a sublink.   }
\label{Fig:Sublink}
\end{figure}

Now by generalizing the method of proof for Theorem \ref{Fibered}, we will  proceed to show that every link in $S^3$ is contained in a hyperbolic link $L$ satisfying Conjecture \ref{VolC}. We show this by representing each link as a closed braid, and we algorithmically add components to obtain $L$.  As mentioned in the introduction,
the proof requires each generator from the braid group to appear; however, since the algorithm works for an unreduced word, it works for every link. Again for a braid $b\in B_n$, we denote its braid closure by $\hat{b}$, and  we define the \textit{length} of $b$ to be the number of standard generators appearing in $b$.

\begin{thm}\label{Sublink}
Let $b\in B_n$ be an unreduced word of length $k$ with each generator appearing at least once. Then $\hat{b}$ is a sublink of a  link  $L$ in $S^3$  where the complement $S^3 \backslash L$ satisfies Conjecture \ref{VolC} and has volume $2kv_8$.
\end{thm}

\begin{proof}
As in Theorem \ref{Fibered}, the maximal tree construction will result in  Hopf links which are not disjoint from each other. By using the Stein map method, we only need to see how to augment the initial shadow surface into a shadow that corresponds to a link in $S^3$. In particular, we  need to cap disks of  our initial  shadow to obtain the $2$-dimensional polyhedron $P_L$ defined in Section \ref{2.4}. For the shadows of the fundamental shadow links this is possible, the augmentation will always result in a way to find the disjoint Hopf links. This is an advantage in using the Stein map approach; however, an alternative proof using the maximal tree construction of the fundamental shadow links will be included.

Let $b \in B_n$ be an unreduced word of length $k$ containing each generator at least once. We can transform the braid closure $\hat{b}$ into a planar $4$-valent graph $G$ by projecting the link onto $S^2$ such that each crossing of $\hat{b}$ becomes a vertex where we keep track of the sign of the crossing. Note that if we do not have at least one of each generator, then $G$ will not be connected. Now we consider the simple polyhedron $S(P)$  of the fundamental shadow link obtained from the  graph $G$ by replacing each vertex, as in Figure \ref{Fig:Graph}, such that the signs of the crossings   match  $\hat{b}$, and we replace each edge with the identity $3$-braid. The  boundary of $S(P)$ are copies of $S^1$. Since the graph $G$ was obtained from a closed braid, one of the boundary components $C$ is outermost in the sense it bounds a disk $D$ such that $D\times (-\epsilon, \epsilon)$ contains the remaining components of $\partial(S(P))$ for some $\epsilon$. Because of this component $C$ and of our choice of using the identity $3$-braids, this shadow can be augmented by capping disks to obtain a shadow of the form $P_L$ given in Section \ref{2.4}. Each of these capped disks will correspond to a disjoint Hopf link. 

Now consider the boundary-decorated  shadowed polyhedron $P$ with simple polyhedron $S(P)$ such that  $C$ is colored as external, the components with crossing information are colored external, and the remaining components are colored as internal. Since we replaced the edges with identity $3$-braids, then the collection of internal colored boundary components are unknotted. Because $\partial(P)$ contains an outermost circle, we can obtain a shadow for a link complement in $S^3$ as in Section \ref{2.4} by capping components by disks. As before in the proof of Theorem \ref{Families}, we  cap $\partial_{int}(P)$ by disks and drill out their corresponding tori to obtain the fundamental shadow link $L$ in $S^3$. Since $G$ was a graph with complexity $k$, then the link complement of $L$ has volume $2kv_8$ and satisfies Conjecture \ref{VolC}.  Now by construction, the collection $\partial_{ext}(P) \backslash \{C\}$ corresponds to the original braid closure $\hat{b}$, therefore $\hat{b} \subset L \subset S^3$. 
\end{proof}

Alternatively, we will provide a proof of Theorem \ref{Sublink} using the maximal tree construction from Section \ref{2.2} and Kirby calculus.

\begin{proof}[Alternate proof of Theorem \ref{Sublink}]
We will  begin with the same $b \in B_n$ of length $k$ which contains each generator at least once, and we will transform the braid closure $\hat{b}$ into a planar $4$-valent graph $G$ where we keep track of the signs of the crossings. As noted in the proof of Theorem \ref{Fibered}, any maximal tree will have the number of edges one less than the number of vertices of $G$. Since the number of vertices is equal to the length of $b$, then any maximal tree will have $(k-1)$ edges. Now note that each vertex of the graph corresponds to a generator of the initial word $b$. We will choose a maximal tree $T$ such that the edges of $T$ satisfy the following two conditions: 

\begin{enumerate}[i).]
\item Each edge connects vertices that  are contained in $G_b$ where $G_b$ is the restriction to $b$ of the projection of $\hat{b}$ to $G$ when $b$ is viewed as a subset of $\hat{b}$. 

\item No edge is adjacent to the outermost region of $G$ when viewed as a graph  in $S^2$. 
\end{enumerate}

Since $b$ contains one of each generator, the subset $G_b$ is connected, and we can choose $(k-1)$ edges which connect all the vertices. 
The second condition  only appears when our unreduced word contains either $(\sigma_1 \sigma_1)$ or $(\sigma_1^{-1} \sigma_1^{-1})$. In either case, there is a choice for an edge of $T$ that satisfies the second condition. Admittedly, the choice of maximal tree is arbitrary; however, this choice of $T$ is consistent with the maximal tree we chose for the proof of Theorem \ref{Fibered}.  For now on, we will view $G$ as a graph in $S^2$, and we define two regions of our graph to be  \emph{adjacent} if there exists an edge in $G \backslash T$ which is adjacent to both of the regions. Notice that  since $T$ does not contain any cycles, any two regions of  $G$ are related by a sequence of adjacent regions. In particular, any region is related to the outermost region of our graph by a sequence of adjacent regions.

Now we consider the fundamental shadow link constructed from the maximal tree such that each vertex  is replaced by the correct sign of the generator in $b$ and each edge is replaced by the identity $3$-braid.  Since each region of the graph $G$ is adjacent to at least one edge of $G \backslash T$, then each corresponding region is punctured by at least one $0$-framed unknot in the link diagram. Note that each region of  $G$ corresponds to an unknotted fundamental shadow link component obtained from replacing the edges by the trivial $3$-braid. Excluding the component obtained from the outermost region, we will denote the collection of the disks these fundamental shadow link components bound by $\{D_i\}_{i \in I}$ with corresponding regions of $G$ denoted by $\{S_i\}_{i \in I}$. The boundary of each of these disks will be the fundamental shadow link component of our disjoint Hopf links.

We will now perform a sequence of handleslides such that each disk $D_i$ is punctured by exactly one $0$-framed unknot. Observe that given any two adjacent regions, there exists a $0$-framed unknot which corresponds to the edge between them. We first consider the collection of regions $\{S_i\}_{i \in I_1}$ of $G$ where $I_1$ is a subset of $I$ such that the regions are adjacent to the outermost region. We will leave alone the corresponding $0$-framed unknots such that they puncture only the collection of disks $\{D_i\}_{i \in I_1}$. Next, we consider the collection $\{S_i\}_{i \in I_2}$ of regions of $G$ which are adjacent to the collection $\{S_i\}_{i \in I_1}$. Note that the corresponding $0$-framed unknots puncture the same previous collection $\{D_i\}_{i \in I_1}$ resulting in pairs of Hopf links which are not disjoint as well as puncturing an additional collection of disks $\{D_i\}_{i \in I_2}$. By performing handleslides of the $0$-framed unknots arising from the edges between  $\{S_i\}_{i \in I_1}$ and  $\{S_i\}_{i \in I_2}$ over the $0$-framed unknots  corresponding to the edges between  $\{S_i\}_{i \in I_1}$ and the outermost region, each $D_{i \in I_1}$ will be punctured by exactly one $0$-framed unknot. Also notice that after the handleslide, both sets of $0$-framed unknots  can be isotoped to a collection of $0$-framed unknots which are parallel to the $z$-axis and each puncturing only one element of $\{D_i\}_{i  \in {I_1 \cup I_2}}$. 

 Now consider the collection of regions  $\{S_i\}_{i \in I_3}$ which are adjacent to  $\{S_i\}_{i \in I_2}$ with $I_3 \cap  I_1 = \emptyset$. The corresponding $0$-framed unknots arising from the edges between the  regions  $\{S_i\}_{i \in I_3}$ and  $\{S_i\}_{i \in I_2}$ will again puncture the disks  $\{D_i\}_{i \in I_2}$  creating  Hopf link pairs which are not disjoint as well as puncturing the collection  $\{D_i\}_{i \in I_3}$. Now we can perform handleslides on the $0$-framed unknots which puncture both the disks  $\{D_i\}_{i \in I_2}$ and  $\{D_i\}_{i \in I_3}$ over the adjacent previous collection of $0$-framed unknots. Since the previous collection of $0$-framed unknots are parallel to the $z$-axis, the $0$-framed unknots arising from the edges between  $\{D_i\}_{i \in I_2}$ and  $\{D_i\}_{i \in I_3}$ after the handleslide are also parallel to the $z$-axis. In addition, they each only puncture one element of $\{D_i\}_{i  \in {I_1 \cup I_2 \cup I_3}}$. Since any two regions are related by a sequence of adjacent regions, we can perform these handleslides until every disk of $\{D_i\}_{i  \in { I}}$ is punctured exactly once by a $0$-framed unknot which is parallel to the $z$-axis. By doing this, the surgery presentation  is the same as the one  obtained by using the Stein map approach where the Hopf links are disjoint from each other. The result now follows from  Remark \ref{Rem:PropOfLink}. This concludes the alternative proof of Theorem \ref{Sublink} using the maximal tree  construction and Kirby calculus.
\end{proof}

In summary, we can construct the link $L$ algorithmically from a given braid closure $\hat{b}$ for $b\in B_n$ where $b$ contains at least one of every generator. We begin by embedding $b$ into $B_{n+1}$ which we denote by $b'$ such that each generator $\sigma_i$ is sent to $\sigma_{i+1}$.  The element $b'$ can be viewed as adding a new strand in front of $b$ which will correspond to the component $C$ in the closure. Now from the braid closure $\hat{b'}$, we will add components to obtain our link $L$. Consider the projection $\pi (\hat{b'})$ onto the surface $S^2$ such that each crossing becomes a vertex. Excluding the regions touching $C$, each region of $\pi(\hat{b'})$ will correspond to a link component of $L$ which is parallel to the $z$-axis. We will now count these regions. We can consider $\pi(\hat{b'})$ as a graph in $S^2$, therefore the Euler characteristic $\chi(S^2) = 2 = v - e + r$ where v,e, and r are  vertices, edges and regions, respectively. Since the vertices correspond to generators, then $v=k$, and since our graph is $4$-valent with the addition of $C$, then $e=2k+1$. This implies that $r=k+3$. Again since we do not need to add link components to regions touching $C$, our link $L$ is obtained from adding links parallel to the $z$-axis in each of the remaining $(k+1)$ corresponding regions of $\hat{b'}$. 

A simple example of the augmentation  can be seen in Figure \ref{Fig:Sublink} on the closure of the unreduced braid   $\sigma_1 \sigma_1^{-1} \sigma_1^{-1} \in B_2$.

\section{Applications to the AMU conjecture}\label{5}

We will now consider applications of the fundamental shadow links from Section \ref{3} towards the AMU conjecture. Given an oriented surface $\Sigma_{g,n}$ where $g$ is the genus and $n$ is the number of boundary components, we consider $Mod(\Sigma_{g,n})$ to be the mapping class group of the surface which fixes the boundary components.


By the Nielsen-Thurston classification of mapping classes, every element of $Mod(\Sigma_{g,n})$ are either periodic, reducible, or pseudo-Anosov. In particular, the type of an element $f\in Mod(\Sigma_{g,n})$ determines the geometric structure of the mapping torus $M(f) = \Sigma_{g,n} \times [0,1] / (x,0) \sim (f(x),1)$. For example, Thurston showed in \cite{Thu88} that the mapping torus $M(f)$ is hyperbolic if and only if $f$ is pseudo-Anosov.

As stated in the introduction,  Andersen, Masbaum, and Ueno conjectured that the geometry of the mapping class group is related to the quantum representation with Conjecture \ref{AMU}. 
We will extend the known results for Conjecture \ref{AMU} to elements in the mapping class groups of $Mod(\Sigma_{g,4})$ for all $g$, $Mod(\Sigma_{n-1,n})$ and $Mod(\Sigma_{n-2,n})$ for $n \geq 5$,  and for the few concrete cases of $Mod(\Sigma_{1,3})$ and $Mod(\Sigma_{3,3})$. We will begin by stating the following where we again mention two  elements $f, g \in Mod(\Sigma_{g,n})$  are \emph{independent} if there is no $h \in Mod(\Sigma_{g,n})$ such that both $f$ and $g$ are  conjugate to non-trivial powers of $h$.

\begin{thm}\label{4Sphere} 
Given $g \geq 0$, there is a pseudo-Anosov element of $Mod(\Sigma_{g,4})$ that satisfies Conjecture \ref{AMU}. Furthermore for $g \geq 3$, there are infinitely many pairwise independent pseudo-Anosov elements in $Mod(\Sigma_{g,4})$ that satisfy Conjecture \ref{AMU}.
\end{thm}

In order to prove Theorem \ref{4Sphere}, we will need some preparation. We will begin by stating  a necessary result by Detcherry and Kalfagianni. In \cite{DetK19}, they  show  Conjecture \ref{VolC} implies  Conjecture \ref{AMU}. In particular, they show the following.

\begin{thm}[\cite{DetK19}, Theorem $1.2$]\label{LimInf}
Let $f \in Mod(\Sigma_{g,n})$ be a pseudo-Anosov mapping class, and let $M(f)$ be the mapping torus of $f$. If $lTV(M(f))>0$, then $f$ satisfies Conjecture \ref{AMU} where 
$$lTV(M(f)) = \lim_{r\to \infty} \inf  \frac{2 \pi}{r} \log| TV_q(M)|$$
 for $r$ odd and $q=\exp\left(\frac{2 \pi i}{r}\right)$. 
\end{thm}

If Conjecture \ref{VolC} is satisfied for a hyperbolic manifold $M(f)$, then  $f$ is pseudo-Anosov and $ lTV(M(f)) =Vol(M(f))  >0$ which implies Conjecture \ref{AMU}.  By using this statement, we will verify the  results shown throughout the remainder of the paper.

 With the techniques developed in \cite{DetK19} and the closed braids from Table \ref{Table:Mon}, we can find infinite families satisfying Conjecture \ref{AMU}. In particular in \cite{DetK19}, they focus on a $2$-parameter family of links which contains the figure-eight as a sublink to create their examples. We will demonstrate similar steps using a $2$-parameter family of links which instead contain the Borromean rings as a sublink. The links in this family can be realized as the closure of an element of a braid group such that if a standard generator of the braid group appears in the braid, it always appears with the same sign. A braid with such a property is known as a \emph{homogeneous braid}.

We will first introduce the notion of a Stallings twist. In \cite{Sta78},  Stallings showed the closure of a homogeneous braid is a fibered link over $S^1$ with fibered surface $\Sigma_{g,n}$ where the surface is obtained from Seifert's algorithm. In particular, the closure of a homogeneous braid is the mapping torus of a surface $\Sigma_{g,n}$ for some element $f \in Mod(\Sigma_{g,n})$.   Stalling also showed an operation which we will state as the \emph{Stallings twist} that takes a fibered link and transforms it into another fibered link with the same fibered surface.  The operation is performing a $1/m$ surgery on a simple closed curve of the fibered surface where the framing of the curve is induced by the normal vector of the surface. In general, the Stallings twist may be trivial in the sense that the link complement obtained afterwards is still homeomorphic to the complement of the original link. In  \cite{DetK19}, they show a criterion to obtain non-trivial Stallings twists. In particular, if we consider our link as the closure of a braid with  the braid containing the element $\sigma_{i}^2 \sigma_{i+1}^{-2}$, we can find a non-trivial Stallings twist on our surface. Now by using the non-trivial Stallings twist, we can obtain infinitely many pairwise independent pseudo-Anosov mapping class group elements as stated in the following.

\begin{thm}[\cite{DetK19}, Theorem $5.4$]\label{StalTwist}
Let $L$ be a hyperbolic fibered link with fiber $\Sigma$ and monodromy $f$. Suppose that $L$ contains a sublink $K$ with $lTV(S^3 \backslash K) >0$. Suppose, moreover, that the fiber $\Sigma$ admits a non-trivial Stallings twist along a curve $c \subset \Sigma$ such that the interior of the twisting disc $D$ intersects $K$ at most once geometrically. Let $\tau_c$ denote the Dehn twist of $\Sigma$ along $c$. Then the family $\{f \circ \tau_c^m\}_m$ of homeomorphisms  gives infinitely many pairwise independent  pseudo-Anosov mapping classes in $Mod(\Sigma)$ that satisfy Conjecture \ref{AMU}.
\end{thm}

\begin{figure}
\labellist
\pinlabel \small{$2m$} at 155 110
\endlabellist
\centering
\includegraphics[scale=.6]{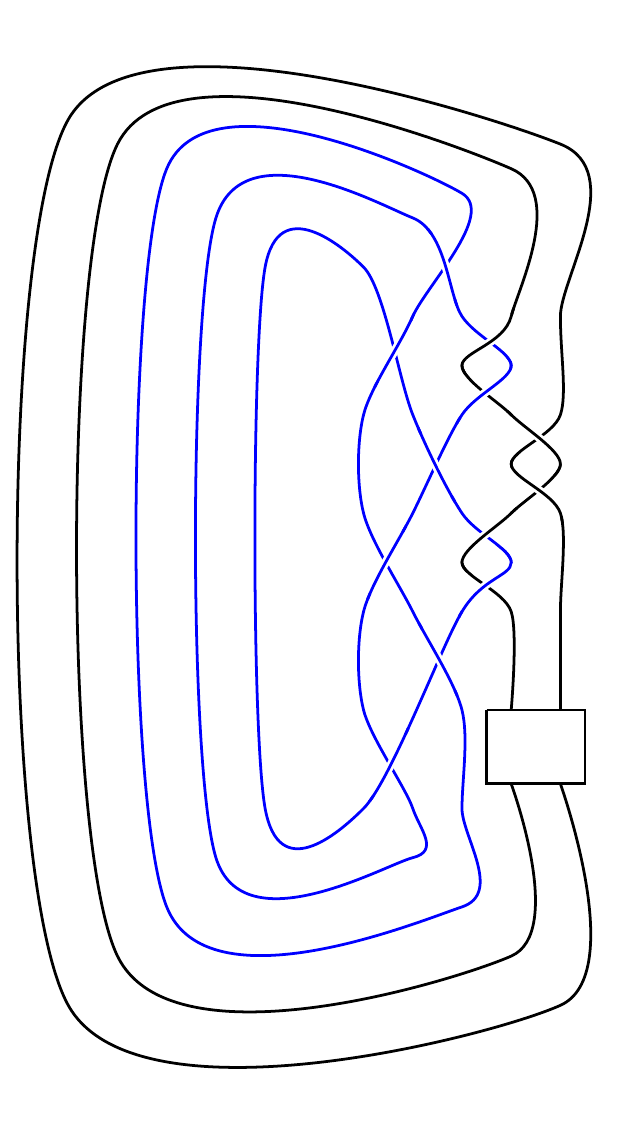}
\caption{The link $L_{5,m}$ where the box indicates $2m$ negative crossings with the Borromean rings as a sublink colored blue.  }
\label{Fig:L5m}
\end{figure}

We will now begin constructing our $2$-parameter family $L_{n,m}$ for $n \geq 4$ and $m \geq 1$ where $n$ is the number of components. Let $L$ be the Borromean rings which corresponds to the link $L_{6a4}$ in Table \ref{Table:Links} and \ref{Table:Mon}. We can see that the Borromean rings are the closure of the homogeneous braid $(\sigma_{2}^{-1} \sigma_{1})^3 \in B_3$. Define $K_i$ to be the component obtained as the closure of the $i$-th strand of the braid element such that $K_1 \cup K_2 \cup K_3 = L$.  

We will begin with the link $L_{5,m}$ which is the union of $L$ with the additional components $K_4$ and $K_5$ shown in Figure \ref{Fig:L5m}. We can see in  Figure \ref{Fig:L5m} that the parameter $m$ indicates the number of negative full-twists. Also notice our braid contains the subword $\sigma_{3}^2 \sigma_{4}^{-2}$.  As mentioned before, this indicates we will be able to obtain a non-trivial Stallings twist.  From $L_{5,m}$, we can add additional link components $K_i$ for $6 \leq i \leq n$ such that the union of the components is our link $L_{n,m}$. To obtain $L_{n+1,m}$ from $L_{n,m}$, we add a strand whose closure is $K_{n+1}$ such that following along $K_{n}$ has $2$ crossings with $K_{n+1}$, then $2$ crossings with $K_{n-1}$, and then again $2$ crossings with $K_{n+1}$ where the sign of the crossing is chosen such that the link $L_{n+1,m}$ is alternating. Each time we add an additional component, we increase the number of crossings by $4$. 

Now for the special case when $n=4$, we replace the $2m$ negative crossings in $L_{5,m}$ with $2m-1$ negative crossings. This results in a $4$-component link $L_{4,m}$.

\begin{prop}\label{NonCon}
Suppose that either $n=4$ and $g \geq 3$, or $g + 2 \geq n \geq 5$, then there are infinitely many pairwise independent pseudo-Anosov elements of $Mod(\Sigma_{g,n})$ that satisfy Conjecture \ref{AMU}.
\end{prop}

\begin{proof}
 The link $L_{n,m}$ contains the Borromean rings for every $n \geq 4$ and $m \geq 1$ as a sublink.   By Corollary $5.3$ in \cite{DetK20},  Detcherry and Kalfagianni show that drilling out tori from a $3$-manifold does not decrease the Turaev-Viro invariant of the resultant manifold. In particular, since the Turaev-Viro invariant of the Borromean rings grows exponentially, then the Turaev-Viro invariant of the link complement of $L_{n,m}$ grows exponentially with  $lTV>0$. By Theorem \ref{LimInf}, for large enough $r$, there exists a quantum representation of the fibered link that is infinite order. Also  Menasco's criterion \cite{Men84}  states any prime non-split alternating diagram of a link that is not the standard diagram of the $T(2,q)$ torus link is a hyperbolic link. This implies that the monodromy is pseudo-Anosov and  Conjecture \ref{AMU} is verified for $L_{n,m}$. Now we can find a simple closed curve on our fibered surface that satisfies the conditions of Theorem \ref{StalTwist}, therefore we have infinitely many pairwise independent pseudo-Anosov elements for each fibered surface obtained from $L_{n,m}$ that satisfy Conjecture \ref{AMU}. 
 
 We will now calculate the resulting fibered surfaces of the link complement of $L_{n,m}$. Every link $L_{n,m}$ is the braid closure of a homogeneous braid, therefore we can calculate the genus of the Seifert surface using the  equation
 \begin{align}\label{Eq:Surface}
g = \frac{2 + C - \text{Braid Index} - n }{2}
\end{align}
where $C$ is the number of crossings. 
 
  First we consider the case when $n=4$. We have that the number of crossings $C=2m+11$ and the braid index is $5$. In this case from Equation (\ref{Eq:Surface}), we have the genus $g=m+2$ for $m \geq 1$. 
 
 Now consider the case for $n \geq 5$. The braid index is equal to $n$ and $C=2m+4n-8$. From this, we calculate $g=m+n-3$ where $m \geq 1$. This results in  our fibered surface $\Sigma_{g,n}$ having either  $n=4$ and $g \geq 3$, or $g+2 \geq n \geq 5$. 
\end{proof}

Proposition \ref{NonCon}  and its proof should be compared to  \cite[Theorem 1.4]{DetK19} where the authors  find infinitely many pairwise independent pseudo-Anosov elements for $n=2$ and $g \geq 3$, or $g\geq n \geq 3$, using monodromies of fibered links 
containing the  figure-eight knot as a sublink. Although the case using the Borromean rings has many overlapping fibered surfaces with using the figure-eight knot, we gain additional infinite families of pairwise independent pseudo-Anosov elements for $Mod(\Sigma_{g,n})$ for certain $g$ and $n$ not previously discussed. In particular, their examples have genus larger than or equal to the number of boundary components; however with the Borromean rings, we obtain infinitely many examples where the genus is strictly less than the number of boundary components.

\begin{table}[h]
\centering
\begin{tabular}{|c|c|c|c|}
\hline
 LinkInfo Name  &  Braid Index & Braid Element  & Fibered Surface\\
 \hline
 \hline
 $L_{6a4}$ & $3$ & $(\sigma_1 \sigma_2^{-1})^3$ & $\Sigma_{1,3}$ \\
 \hline
 $L_{6a4}$ & $3$ & $(\sigma_1^{-1} \sigma_2)^3$ & $\Sigma_{1,3}$ \\
 \hline
 $L_{8n7}$ & $4$ & $\sigma_1 \sigma_2^{-2} \sigma_1 \sigma_3^{-1} \sigma_2^{-2} \sigma_3^{-1}$& $\Sigma_{1,4}$  \\
 \hline
 $L_{10n87}$ & $3$ & $\sigma_1^3 \sigma_2^{2} \sigma_1^2 \sigma_2^2 \sigma_1$& $\Sigma_{3,3}$  \\
 \hline
 $L_{10n97}$ & $4$ & $\sigma_1^{-3} \sigma_2^{-2} \sigma_1^{-1} \sigma_3^{-1} \sigma_2^{-2} \sigma_3^{-1} $ & $\Sigma_{2,4}$ \\
 \hline
 $L_{10n108}$ & $4$ & $(\sigma_1^{-1} \sigma_2^{-1})^3 \sigma_3 \sigma_2^{-2} \sigma_3$ & $\Sigma_{2,4}$  \\
 \hline
 $L_{11n385}$ & $4$ & $(\sigma_1^{-1} \sigma_2^{-1})^3 \sigma_3 \sigma_2^{-2} \sigma_3 \sigma_2^{-1}$ & $\Sigma_{3,3}$  \\
 \hline
\end{tabular}
\caption{Fibered surfaces of links}
\label{Table:Mon}
\end{table}

\begin{cor}\label{Mon}
The monodromy of the closed braids in Table \ref{Table:Mon} obtained from Stalling's fibration is pseudo-Anosov and satisfies Conjecture \ref{AMU}.
\end{cor}

\begin{proof}
From Theorem \ref{TableLinks}, the closed braids in Table \ref{Table:Mon} have homeomorphic complements to the fundamental shadow links, therefore they are hyperbolic and  satisfy Conjecture \ref{VolC}. Since they are hyperbolic,  the monodromy of their fibering as  mapping tori will be pseudo-Anosov. From Theorem \ref{LimInf}, since they satisfy Conjecture \ref{VolC}, then their monodromies also satisfy Conjecture \ref{AMU}. Because the fibered surface is obtained from Seifert's algorithm, the genus $g$ can be computed by Equation (\ref{Eq:Surface}) in Proposition \ref{NonCon} where the number of boundary components $n$ of the surface is the number of components of the closed braid.
\end{proof}

\begin{remark}\label{Rem:Menasco}
Since we used Menasco's criterion, we chose to use the Borromean rings example from Table \ref{Table:Mon}. This is because the family $L_{n,m}$ will be alternating and be guaranteed to be hyperbolic with pseudo-Anosov monodromy. We could have done the previous calculation with any other homogeneous braid from Table \ref{Table:Mon}; however, the resulting manifold may not be hyperbolic. We know the Turaev-Viro invariant of these links grows exponentially and is bounded above by the Gromov norm, therefore the associated monodromy is either pseudo-Anosov or reducible with pseudo-Anosov pieces. In this setting,  additional infinite families can be constructed from the remaining links in Table \ref{Table:Mon} that satisfy Conjecture \ref{AMU}; however, they may only contain pseudo-Anosov pieces.
\end{remark}

Now we have all the necessary components in order to prove Theorem \ref{4Sphere}.

\begin{proof}[Proof of Theorem \ref{4Sphere}]
The combination of Proposition \ref{NonCon} and Corollary \ref{Mon} provide a pseudo-Anosov element of $Mod(\Sigma_{g,4})$ for every $g \geq 1$. In particular, we see from Proposition \ref{NonCon} that for $g\geq 3$, there are infinitely many pairwise independent pseudo-Anosov elements satisfying Conjecture \ref{AMU}. For the remaining case, Andersen, Masbaum, and Ueno in \cite{AndMU}   verified the conjecture for every pseudo-Anosov element in $Mod(\Sigma_{0,4})$ as noted in the introduction. 
\end{proof}

\subsection{Surfaces of genus $0$}\label{5.1}

In addition to Theorem \ref{4Sphere}, we can find examples in genus $0$ surfaces with at least $4$ boundary components that satisfy Conjecture \ref{AMU}. In order to state the result, we recall the following. We say an element $b \in B_m$ is a \emph{pure braid} if the beginning and the ending of each strand are in the same position. Under the correspondence of the mapping class group $MCG(D_m)$ with the braid group $B_m$, if we  consider the boundary of $D_m$ to be an additional puncture, then $\Gamma(b)$ as an element of the pure braid group fixes all $(m+1)$-punctures of $D_m$.  In particular, we can identify $b$ as an element of $Mod(\Sigma_{0,m+1})$. Under this identification, we define the  family of pure braids $\{\omega_m\}$ for  $m \geq 4$ where $\{b_k\}$ is the braid family from Theorem \ref{Fibered}:
 \begin{itemize}
\item For $m=4$, define $\omega_4 := b_1 \in Mod(\Sigma_{0,4})$.
\item For $m \geq 6$ and even, define $\omega_m := b_{m-4} \in  Mod(\Sigma_{0,m}).$
\item For $m=5$ and $m=7$, define $\omega_m := \omega_{m-1}' \in Mod(\Sigma_{0,m})$ where $\omega_{m-1} '$ is an embedding of $\omega_{m-1} \in Mod(\Sigma_{0,m-1})$ into $Mod(\Sigma_{0,m})$ sending each generator $\sigma_j$ to $\sigma_j$. 
\item For $m=9$, define $\omega_{9} \in Mod(\Sigma_{0,9})$ by
 $$\omega_9:=(\sigma_4 \sigma_3^2 \sigma_4)(\sigma_5 \sigma_4 \sigma_3 \sigma_2^2 \sigma_3 \sigma_4 \sigma_5)(\sigma_6 \sigma_5 \sigma_4 \sigma_3 \sigma_2 \sigma_1^2 \sigma_2^{-1} \sigma_3 \sigma_4 \sigma_5 \sigma_6)(\sigma_7 \sigma_6 \sigma_5 \sigma_4 \sigma_3 \sigma_2^2  \sigma_3^{-1} \sigma_4 \sigma_5 \sigma_6 \sigma_7).$$
\item For $m=7+2s$ and $s \geq 2$, define $\omega_{m}\in  Mod(\Sigma_{0,m})$ by
 $$ \omega_m:= \omega_9' \prod_{i=2}^{s} (\sigma_{4+2i} \sigma_{3+2i} \cdots \sigma_4 \sigma_3^2 \sigma_4 \cdots \sigma_{3+2i} \sigma_{4+2i})(\sigma_{5+2i} \sigma_{4+2i} \cdots \sigma_3 \sigma_2^2 \sigma_3^{-1} \sigma_4 \sigma_5 \cdots \sigma_{4+2i}\sigma_{5+2i})$$
where $\omega_9'$ is an embedding of $\omega_9 \in Mod(\Sigma_{0,9})$ into $Mod(\Sigma_{0,m})$ sending each generator $\sigma_j$ to $\sigma_j$.
\end{itemize}

\begin{cor}\label{Even}
For $n \geq 4$, the  element $\omega_n \in Mod(\Sigma_{0,n})$  satisfies Conjecture \ref{AMU}. Additionally  if $n \not \in \{5,7\} $, then $\omega_n$ is pseudo-Anosov. 
\end{cor}

As mentioned previously, Egsgaard and Jorgensen, and separately Santharoubane, verified the existence of elements of the mapping class group $Mod(\Sigma_{0,n})$  satisfying Conjecture \ref{AMU} in \cite{EgsJ16, San17} for when $n$ is even.  By using Theorems \ref{Fibered} and \ref{Sublink}, the examples provided in this paper  can be viewed as explicit elements coming from the mapping class group. In addition to the elements being explicit, we also obtain pseudo-Anosov elements in $Mod(\Sigma_{0,2l+1})$ for $l \geq 4$. As noted in the introduction, these are the first pseudo-Anosov elements obtained in $Mod(\Sigma_{0,n})$  that satisfy Conjecture \ref{AMU} for an odd $n$.  We also remark that Andersen, Masbaum, and Ueno in \cite{AndMU}   verified the conjecture for every pseudo-Anosov element in $Mod(\Sigma_{0,4})$, therefore we could have chosen any pseudo-Anosov element in the $n=4$ case.

In order to prove  Corollary \ref{Even}, we will need to use the topological quantum field theory for colored links in a cobordism where the details can be found in \cite{BlaHMV95}. We denote the Reshetikhin-Turaev invariant for closed manifolds with an embedded colored link $(L,c)$ by $RT_r(M,(L,c))$ where $c$ is an assignment of components of $L$ with elements   of $ U_r=\{0,2, \dots, r-3\}$. The following   relates the traces of the quantum representations to the invariants of the colored link in a closed manifold.

\begin{thm}[\cite{BlaHMV95}]\label{RTTrace}
For $r \geq 3$ and odd, let $\Sigma_{g,n}$ be a compact oriented surface with a coloring of the boundary from elements of $U_r$. Define $\overline{\Sigma_{g,n}}$ to be the surface obtained from $\Sigma_{g,n}$ by capping the components of $\partial \Sigma_{g,n}$ by disks. For $f \in Mod(\Sigma_{g,n})$, let $\overline{f} \in Mod(\overline{\Sigma_{g,n}})$ be the mapping class of the extension of $f$ from capping the disks by the identity. Now let $L\subset M(\overline{f})$ whose components consists of the cores of the tori in $M(\overline{f})$ arising from the capping disks in the mapping tori, then  
$$Tr(\rho_{r,c}(f))= RT_r(M(\overline{f}),(L,c)).$$
\end{thm}

From the Theorem \ref{RTTrace}, Detcherry and Kalfagianni relate the Turaev-Viro invariant to the traces of the quantum representations. This result can be found in the proof of Theorem $1.2$ in \cite{DetK19}.

\begin{prop}[\cite{DetK19}, Theorem $1.2$]\label{TVTrace}
Let $f \in Mod(\Sigma_{g,n})$ and $M(f)$ the mapping torus of $f$, then 
$$TV_q(M(f))= \sum_c |Tr(\rho_{r,c}(f)|^2.$$
\end{prop}

Now by explicitly using the relationship between Conjecture \ref{VolC} and Conjecture \ref{AMU} shown in \cite{DetK19}, we will proceed with the proof for Corollary \ref{Even}.

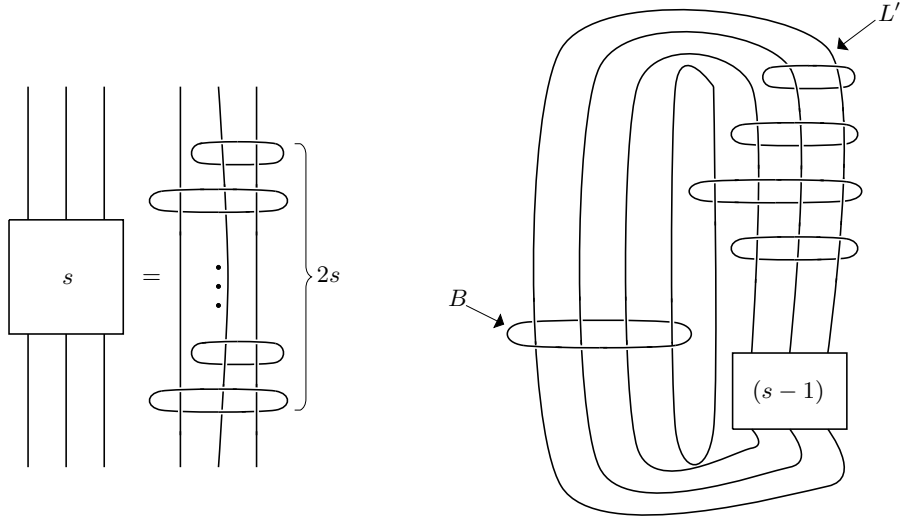
\begin{figure}
\labellist
\pinlabel \small{$L'$} at 443 269
\pinlabel \small{$(s-1)$} at 392 78
\pinlabel \small{$B$} at 226 125
\pinlabel \small{$=$} at 72 135
\pinlabel \small{$s$} at 30 135
\pinlabel \small{$2s$} at 161 137
\endlabellist
\centering
\includegraphics[scale=.75]{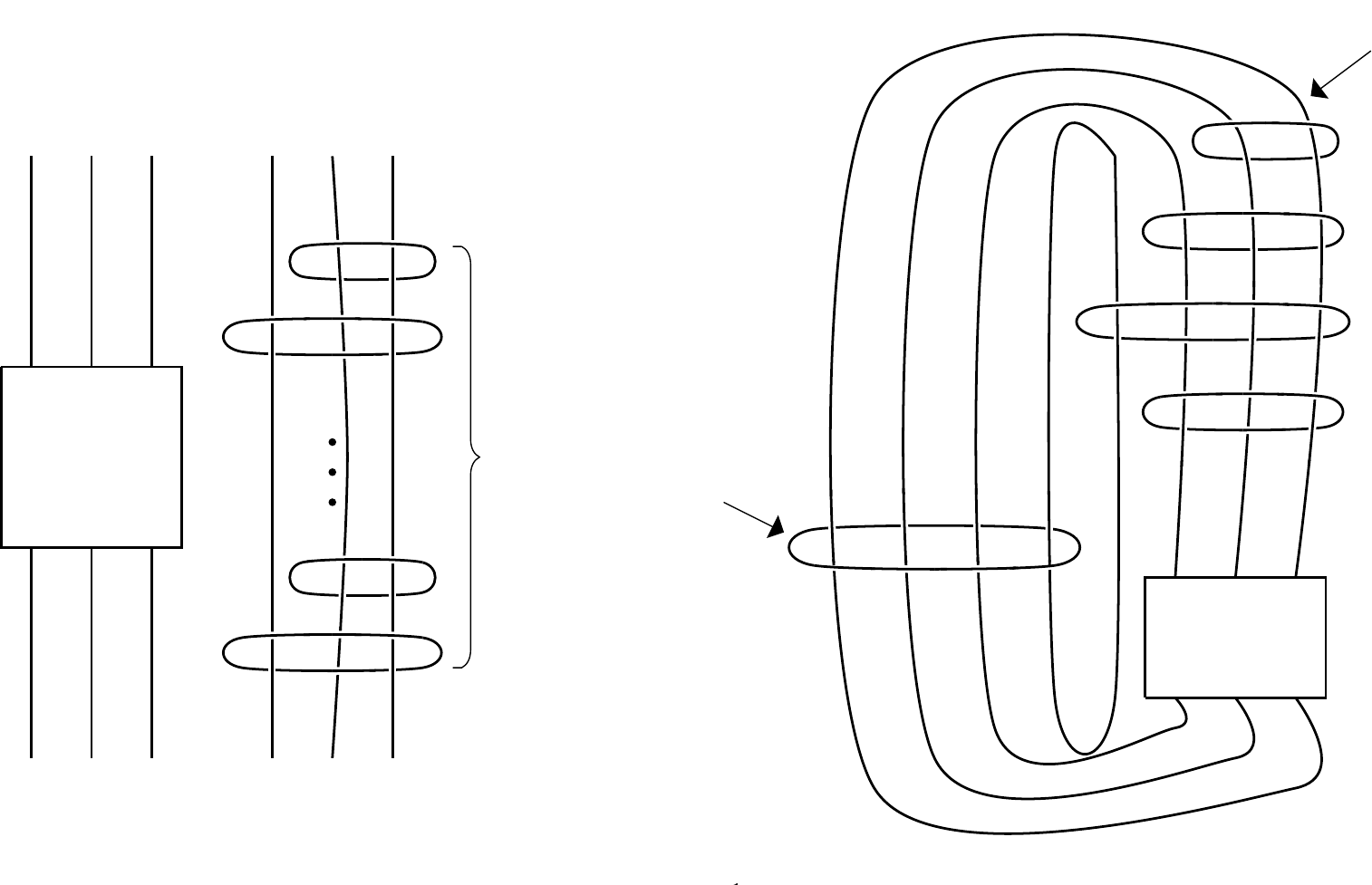}
\caption{On the left, we define the boxed diagram as the identity $3$-braid  with $2s$ additional components. We define the diagram to be the identity $3$-braid without additional components for a box labeled with $0$. On the right, we show the fundamental shadow link in $S^3$, up to ambient isotopy,  obtained from the algorithm in Theorem \ref{Sublink} for the element  $ (\sigma_1 \sigma_1^{-1})^s (\sigma_2 \sigma_2^{-1} ) \in B_3$. The braid axis is labeled $B$, the link component that bounds a $(3+2s)$-punctured disk is labeled $L'$, and the link contains $(7+2s)$ components.}
\label{Fig:3Braid}
\end{figure}

\begin{proof}[Proof of Corollary \ref{Even}]
First we will begin with showing the case for our genus $0$ surface of even boundary components.  We note that for every $m\geq 3$ and odd, there exists a manifold fibering over $S^1$ that has fibered surface $D_m$ whose complement satisfies Conjecture \ref{VolC}. The examples of such manifolds can be obtained from the manifolds $M_{1}$ and $M_k$ for $k$ even in Theorem \ref{Fibered}. These manifolds are the complements of a hyperbolic link in $S^3$ with fibered surface $D_3$ and $D_{k+3}$ with volumes $2v_8$ and $2kv_8$, respectively. 

Fix $M_k$ to be one of these  manifolds, and let $\Gamma(b_k)$ be its corresponding monodromy. Under the correspondence of $MCG(D_{m})$ with the braid group $B_{m}$, $\Gamma(b_k)$ represents an element of the pure braid group such that $b_k$ can be identified as an element in $Mod(\Sigma_{0,m+1})$ with $M(b_k)$ homeomorphic to $M_k$

Since  our manifold $M_k$ satisfies Conjecture \ref{VolC} and is homeomorphic to $M(b_k)$, in particular we have $Vol(M_k)>0$ and $TV_q(M(b_k)) = TV_q(M_k)$. Now by  Proposition \ref{TVTrace}, we have the following equalities.
\begin{align*}
\lim_{r\to \infty}   \frac{2 \pi}{r} \log |\sum_c |Tr(\rho_{r,c}(b_k))|^2 | &=\lim_{r\to \infty}   \frac{2 \pi}{r} \log |TV_q(M(b_k))|=\lim_{r\to \infty}   \frac{2 \pi}{r} \log| TV_q(M_k)|\\ 
&= Vol(M_k) >0. 
\end{align*}

This implies that the traces of the quantum representations must grow exponentially as $r \to \infty$. Now in \cite{BlaHMV95}, it is shown that the dimensions of the vector spaces of the quantum representations only grows polynomially with $r$, therefore there exists a large enough $r$ such that 
$|Tr(\rho_{r,c}(b_k))|$ exceeds the dimension. This means for some color $c$, there is a representation $\rho_{r,c}(b_k)$ such that at least one eigenvalue has modulus larger than $1$. Therefore $\rho_{r,c}(b_k)$ is infinite order. 

Since $M_k$ is a fundamental shadow link, then $M(b_k)$ is hyperbolic and $b_k \in Mod(\Sigma_{0,m+1})$ is pseudo-Anosov with $m \geq 3$ and odd. By the previous argument, we can find a quantum representation of $b_k$ that has infinite order for large enough $r$. This shows for any $n\geq 4$ and even, we can find a pseudo-Anosov element $b_k \in Mod(\Sigma_{0,n})$ that satisfies Conjecture \ref{AMU}.

Now in a similar way, we can consider the case for surfaces of genus $0$ and with an odd number of boundary components. By Corollary \ref{Gromov}, we can find explicit elements in the mapping class group of $D_m$ where $m\geq 4$ and even such that their asymptotics grow exponentially. The monodromies of these manifolds are the same as above, except that we consider them the elements of a mapping class group with more boundary components. In terms of braid groups, this is equivalent to adding additional strands. As before, this implies that for a large enough $r$, there exists a quantum representation that has infinite order. In these cases, the monodromies will not be pseudo-Anosov; however, they will contain pseudo-Anosov pieces and  satisfy Conjecture \ref{AMU}.

Finally we show the case when $n \geq 9$ and odd. Consider $ (\sigma_1 \sigma_1^{-1})^s (\sigma_2 \sigma_2^{-1} ) \in B_3$ as the unreduced word of length $2s +2$ for $s \geq 1$. From Theorem \ref{Sublink}, we can construct a hyperbolic manifold of volume $2(2s+2)v_8$ satisfying Conjecture \ref{VolC} which is the complement of the  link shown in the right diagram of Figure \ref{Fig:3Braid} where the box labeled with $(s-1)$ is defined on the left diagram. We note that when $(s-1)=0$, the boxed diagram will be defined to be the identity $3$-braid without additional components. As in the proof of Theorem \ref{Fibered}, we can perform a full-twist along the component $L'$ to obtain a braided link $\overline{\omega_{7+2s}}$ with braid axis $B$ where $\omega_{7+2s}$ is a pure braid. Since $L'$ and $B$ bound a $(3+2s)$ and $4$-punctured disk, respectively, then the resultant link will have fibered surface $D_{6+2s}$. The elements in the pure braid group obtained from the full-twist  are  defined  explicitly as follows:

 \begin{itemize}
\item For $s=1$, define $\omega_{7+2s} \in B_8$ by
 $$\omega_{7+2s}:=(\sigma_4 \sigma_3^2 \sigma_4)(\sigma_5 \sigma_4 \sigma_3 \sigma_2^2 \sigma_3 \sigma_4 \sigma_5)(\sigma_6 \sigma_5 \sigma_4 \sigma_3 \sigma_2 \sigma_1^2 \sigma_2^{-1} \sigma_3 \sigma_4 \sigma_5 \sigma_6)(\sigma_7 \sigma_6 \sigma_5 \sigma_4 \sigma_3 \sigma_2^2  \sigma_3^{-1} \sigma_4 \sigma_5 \sigma_6 \sigma_7).$$
\item For $s\geq 2$, define $\omega_{7+2s}\in  B_{6+2s}$ by
 $$ \omega_{7+2s}:= \omega_9' \prod_{i=2}^{s} (\sigma_{4+2i} \sigma_{3+2i} \cdots \sigma_4 \sigma_3^2 \sigma_4 \cdots \sigma_{3+2i} \sigma_{4+2i})(\sigma_{5+2i} \sigma_{4+2i} \cdots \sigma_3 \sigma_2^2 \sigma_3^{-1} \sigma_4 \sigma_5 \cdots \sigma_{4+2i}\sigma_{5+2i})$$
where $\omega_9'$ is an embedding of $\omega_9 \in Mod(\Sigma_{0,9})$ into $Mod(\Sigma_{0,m})$ sending each generator $\sigma_j$ to $\sigma_j$.
\end{itemize}

Since $\omega_{7+2s} \in B_{6+2s}$  is a pure braid and the complement of the braided link $\overline{\omega_{7+2s}}$ satisfies Conjecture \ref{VolC}, we can use the same argument as the case when $n$ was even. Therefore for any $s \geq 1$, we can find a corresponding element in $Mod(\Sigma_{0,7+2s})$  which is pseudo-Anosov and satisfies Conjecture \ref{AMU}.

\end{proof}

\Addresses


\begin{thebibliography}{9}
\bibitem{Ale23} 
J. W. Alexander.  
``A Lemma on Systems of Knotted Curves". 
In: \textit{Proc. Nat. Acad. Sci. USA}  
9.3 (1923), pages 93-95.    


\bibitem{AndMU} 
J.E. Andersen,  G. Masbaum, and K. Ueno.  
``Topological Quantum Field Theory and the Nielsen-Thurston classification of $M(0,4)$". 
In: \textit{Mathematical Proceedings of the Cambridge Philosophical Society}  
141.3 (2006), pages 477-488.    

\bibitem{Bak02} 
M. Baker.  
``All links are sublinks of arithmetic links". 
In: \textit{Pacific Journal of Mathematics}  
203.2 (2002), pages 257-263.    

\bibitem{Bel20} 
G. Belletti.  
``A maximum volume conjecture for hyperbolic polyhedra". 
In: \textit{arXiv preprint arXiv:2002.01904}  
(2020).    

\bibitem{BelDKY} 
G. Belletti,  R. Detcherry,  E. Kalfagianni, and T. Yang.  
``Growth of quantum $6j$-symbols and applications to the Volume Conjecture". 
In: \textit{Journal of Differential Geometry}  
(to appear).    

\bibitem{BlaHMV95} 
C. Blanchet, N. Habegger, G. Masbaum, and P. Vogel.  
``Topological Quantum Field Theories derived from the Kauffman bracket". 
In: \textit{Topology}  
34.4 (1995), pages  883-927.    

\bibitem{CheY18} 
Q. Chen and T. Yang.  
``Volume conjectures for the Reshetikhin-Turaev and the Turaev-Viro invariants". 
In: \textit{Quantum Topology}  
9.3 (2018), pages 419-460.    

\bibitem{ConT08} 
F. Constantino and D. Thurston.  
``$3$-manifolds efficiently bound $4$-manifolds". 
In: \textit{Journal of Topology}  
1.3 (2008), pages 703-745.    

\bibitem{SnapPy} 
Marc Culler, Nathan M. Dunfield,  Matthias Goerner, and Jeffrey R. Weeks.  
\textit{Snap{P}y, a computer program for studying the geometry and topology of $3$-manifolds}. 
Available at http://snappy.computop.org (01/12/2018)

\bibitem{DetK202} 
R. Detcherry and E. Kalfagianni.  
``Cosets of monodromies and quantum representations". 
In: \textit{arXiv preprint arXiv:2001.04518}  
(2020).    

\bibitem{DetK20} 
R. Detcherry and E. Kalfagianni.  
``Gromov norm and Turaev-Viro invariants of $3$-manifolds". 
In: \textit{Ann. Sci. de l'Ecole Normale Sup.}  
(to appear).    

\bibitem{DetK19} 
R. Detcherry and E. Kalfagianni.  
``Quantum representations and monodromies of fibered links". 
In: \textit{Advances in Mathematics}  
351 (2019), pages 676-701.    

\bibitem{DetKY18} 
R. Detcherry,  E. Kalfagianni, and T. Yang.  
``Turaev-Viro invariants, colored Jones polynomial and volume". 
In: \textit{Quantum Topology}  
9.4 (2018), pages 775-813.    

\bibitem{EgsJ16} 
J.K. Egsgaard and S.F. Jorgensen.  
``The homological content of the Jones representations at $q=-1$". 
In: \textit{Journal of Knot Theory and Its Ramifications}  
25.11 (2016), 25 pages.    

\bibitem{FarM12} 
B. Farb and D. Margalit.  
\textit{A primer on mapping class groups}. 
Vol. 49, Princeton Mathematical Series. Princeton University Press, 2012.

\bibitem{IshK17} 
M. Ishikawa and Y. Koda.  
``Stable maps and branched shadows of $3$-manifolds". 
In: \textit{Mathematische Annalen}  
367.3 (2017), pages 1819-1863.    

\bibitem{Kas97} 
R. M. Kashaev.  
``The Hyperbolic volume of knots from the quantum dilogarithm". 
In: \textit{Lett. Math. Phys}  
39.3 (1997), pages 269-275.    

\bibitem{Lev85} 
H. Levine.  
\textit{Classifying immersions into $\mathbb{R}^4$ over stable maps of $3$-manifolds into $\mathbb{R}^2$}. 
Vol. 1157, Lecture Notes in Mathematics. Springer, Berlin, 1985.

\bibitem{LinkInfo} 
C. Livingston and A. H. Moore.  
\textit{KnotInfo: Table of Link Invariants}. 
http://www.indiana.edu/~knotinfo (May 18, 2020)

\bibitem{MarS16} 
J. March\'e and R. Santharoubane.  
``Asymptotics of quantum representations of surface groups". 
In: \textit{arXiv preprint arXiv:1607.00664}  
(2016).    

\bibitem{Men84} 
W. Menasco.  
``Closed incompressible surfaces in alternating knot and link complements". 
In: \textit{Topology}  
23.1 (1984), pages 37-44.    

\bibitem{MurM01} 
H. Murakami and J. Murakami.  
``The colored Jones polynomials and the simplicial volume of a knot". 
In: \textit{Acta Mathematica}  
186.1 (2001), pages 85-104.    

\bibitem{Oht18} 
T. Ohtsuki.  
``On the asymptotic expansion of the quantum $SU(2)$ invariant at $q=\exp(4 \pi \sqrt{-1}/ N)$ for closed hyperbolic $3$-manifolds obtained by integral surgery along the figure-eight knot". 
In: \textit{Algebraic \& Geometric Topology}  
18.7 (2018), pages 4187-4274.    

\bibitem{ResT91} 
N. Reshetikhin and V. Turaev.  
``Invariants of $3$-manifolds via link polynomials and quantum groups". 
In: \textit{Invent. Math.}  
103.3 (1991), pages 547-597.   

\bibitem{Sae96} 
O. Saeki.  
``Simple stable maps of $3$-manifolds into surfaces". 
In: \textit{Topology}  
35.3 (1996), pages 671-698.    

\bibitem{San17} 
R. Santharoubane.  
``Action on $M(0,2n)$ on some kernel spaces coming from $SU(2)$-TQFT". 
In: \textit{J. Lon. Math. Soc}  
95.3 (2017), pages 785-803.    

\bibitem{San12} 
R. Santharoubane.  
``Limits of quantum $SO(3)$ representations for the one-holed torus". 
In: \textit{Journal of Knot Theory and Its Ramifications}  
21.11 (2012).    

\bibitem{Sta78} 
J. R. Stallings.  
``Constructions of fibered knots and links". 
In: \textit{Proceedings of Symposia in Pure Mathematics}  
32 (1978), pages 55-60.    

\bibitem{Thu88} 
W. Thurston.  
``On the geometry and dynamics of diffeomorphisms of surfaces". 
In: \textit{Bull. Amer. Math. Soc.}  
19.2 (1988), pages 417-431.    

\bibitem{Tur94} 
V. Turaev.  
\textit{Quantum invariants of knots and $3$-manifolds}. 
de Gruyter Studies in Mathematics. de Gruyter, 1994.

\bibitem{TurV92} 
V. Turaev and O. Viro.  
``State sum invariants of $3$-manifolds and quantum $6j$-symbols". 
In: \textit{Topology}  
31.4 (1992), pages 865-902.    

\bibitem{Rol08} 
R. Van Der Veen.  
``Proof of the volume conjecture for Whitehead chains". 
In: \textit{Acta. Math. Vietnamica}  
33.3 (2008), pages 421-431.    

\bibitem{Rol09} 
R. Van Der Veen.  
``The volume conjecture for augmented knotted trivalent graphs". 
In: \textit{Algebraic \& Geometric Topology}  
9.2 (2009), pages 691-722.    

\bibitem{Won19} 
K. H. Wong.  
``Asymptotics of some quantum invariants of the Whitehead chains". 
In: \textit{arXiv preprint arXiv:1912.10638}  
(2019).    

\bibitem{WonY20} 
K. H. Wong, and T. Yang.  
``On the Volume Conjecture for hyperbolic Dehn-filled $3$-manifolds along the figure-eight knot". 
In: \textit{arXiv preprint arXiv:2003.10053}  
(2020).    

\end{thebibliography}
\end{document}